\documentclass[11pt,a4paper,bibliography=totocnumbered,listof=totocnumbered]{article}
\usepackage[ngerman,english]{babel}
\usepackage[T1]{fontenc}
\usepackage[utf8]{inputenc}
\usepackage{amsmath}
\usepackage{mathtools}
\usepackage{amsthm}
\usepackage{amsfonts}
\usepackage{amssymb}
\usepackage{graphicx} 
\usepackage{geometry} 
\usepackage{setspace} 
\usepackage{subfig} 

\usepackage{paralist} 
\usepackage{enumitem}

\usepackage{titlesec} 
\usepackage[colorlinks, linkcolor = blue, citecolor = blue, urlcolor = blue]{hyperref}
\usepackage[toc,page]{appendix}
\usepackage{cite}
\usepackage{bbm}
\usepackage{dsfont}

\newcommand{\mE}{\mathcal{E}}

\DeclareMathOperator{\mH}{\mathcal{H}}

\DeclareMathOperator{\supp}{supp}

\newtheoremstyle{normal}
{10pt}
{10pt}
{}
{}
{\bfseries}
{}
{0em}
{\bfseries{\thmname{#1}\thmnumber{ #2}\thmnote{ \hspace{0em}(#3)\newline}}}

\newtheoremstyle{standard}  
  {10pt}   
  {}   
  {\itshape}  
  {}       
  {\bfseries} 
  {:}         
  {0.2cm}  
  {}          

\newtheoremstyle{mittitel}  
  {10pt}   
  {}   
  {\itshape}  
  {}       
  {\bfseries} 
  {:}         
  {0.2cm}  
  {\bfseries{\thmname{#1}\thmnumber{ #2}\thmnote{ \hspace{0em}(#3)\newline}}}          

\title{Stochastic Wave Equations defined by Fractal Laplacians on Cantor-like Sets\vspace{-1ex} }
\author{Tim Ehnes\footnote{ Institute of Stochastics and Applications, University of Stuttgart, Pfaffenwaldring 57, 70569 Stuttgart, Germany, e-mail: tim.ehnes@mathematik.uni-stuttgart.de }}
\date{\vspace{-5ex}}

\begin{document}
\maketitle

\setlength{\parindent}{10pt}

\titlespacing{\section}{0pt}{20pt plus 4pt minus 2pt}{14pt plus 2pt minus 12pt}
\titlespacing{\subsection}{0pt}{20pt plus 4pt minus 2pt}{8pt plus 2pt minus 2pt}
\titlespacing{\subsubsection}{0pt}{12pt plus 4pt minus 2pt}{-6pt plus 2pt minus 2pt}

\theoremstyle{standard}

\newtheorem{thm}{Theorem}[section] 
\newtheorem{satz}[thm]{Proposition} 
\newtheorem{lem}[thm]{Lemma}
\newtheorem{kor}[thm]{Corollary} 
\newtheorem{defi}[thm]{Definition} 
\newtheorem{bem}[thm]{Remark}
\newtheorem{hyp}[thm]{Assumption}
\newtheorem{exa}[thm]{Example}

\begin{abstract}
 We study stochastic wave equations in the sense of Walsh defined by fractal Laplacians on Cantor-like sets. For this purpose, we give an improved estimate on the uniform norm of eigenfunctions and approximate the wave propagator using the resolvent density. Afterwards, we establish existence and uniqueness of mild solutions to stochastic wave equations provided some Lipschitz and linear growth conditions. We prove Hölder continuity in space and time and compute the Hölder exponents. Moreover, we are concerned with the phenomenon of weak intermittency. 
\end{abstract}

\section{Introduction}\label{section_1}

In this paper we study second-order hyperbolic stochastic partial differential equations defined by generalized second order differential operators.
To introduce the operator of interest, let $[a,b]\subset \mathbb{R}$ be a finite interval, $\mu$ a finite non-atomic Borel measure on $[a,b]$, $\mathcal{L}^2([a,b],\mu)$ the space of measurable functions $f$ such that $\int_a^b f^2d\mu<\infty$ and $L^2([a,b],\mu)$ the corresponding Hilbert space of equivalence classes with inner product $\langle f,g\rangle_{\mu}\coloneqq \int_a^b fgd\mu$. We define
\begin{align*}
\mathcal{D}_{\mu}^2\coloneqq \Big\lbrace f\in C^1((a,b))\cap C^0([a,b]): &\exists \left(f^{\prime}\right)^{\mu}\in\mathcal{L}_2([a,b],\mu): \\ &f^{\prime}(x)=f^{\prime}(a)+\int_a^x \left(f^{\prime}\right)^{\mu}(y)d\mu(y), ~~ x\in[a,b]\Big\rbrace.
\end{align*}
The Krein-Feller operator with respect to $\mu$ is given as 
\begin{align*}
\Delta_{\mu}: \mathcal{D}_{\mu}^2\subseteq L^2([a,b],\mu)\to L^2([a,b],\mu),~~ f\to\left(f^{\prime}\right)^{\mu}.
\end{align*}
This operator has been introduced, for example, in \cite{FAF,IKD,KS,KO,LG}, especially as the infinitesimal generator of a so-called Quasi diffusion. It is a measure-theoretic generalization of the classical second weak derivative $\Delta_{\lambda^1}$, where $\lambda^1$ is the one-dimensional Lebesgue measure. \\

We recall the well-known physical motivation for Krein-Feller operators (see \cite[Section 1.2]{AE}): We consider a flexible string of length $1$ clamped between two points $x=0$ and $x=1$ such that, if we deflect it, a tension force drives it back towards its state of
equilibrium.
The mass distribution of the bar shall have a density denoted by $\rho: [0,1]\to\mathbb{R}$. For reasons of simplicity, we assume that for the tangentially acting tension force $F$ it holds $F=1$. Then, the deviation of the string, the function $u(t,x)$, is determined by the wave equation
\begin{align}
\kappa \frac{\partial^2 u}{\partial x^2}(t,x)=c\rho(x)\frac{\partial u}{\partial t}(t,x)\label{wave_equation_density}
\end{align}
with Dirichlet boundary conditions $u(t,0)=u(t,1)=0$ for all $t\geq 0$. We impose Neumann boundary conditions
$\frac{\partial u}{\partial x}(t,0)=\frac{\partial u}{\partial x}(t,1)=0$ if the ends of the strings are attached to a pair of frictionless tracks which are free to move up and down..
In order to solve this wave equation, we use the separation of variables and write 
$u(t,x)=f(x)g(t)$, which yields
\begin{align*}
\kappa f^{\prime\prime}(x)g(t)=c\rho(x)f(x)g^{\prime\prime}(t)
\end{align*}
and by resorting
\begin{align*}
\frac{f^{\prime\prime}(x)}{\rho(x)f(x)}=\frac{c}{\kappa}\frac{g^{\prime\prime}(t)}{g(t)} 
\end{align*}
for all $t$ and $x$. Consequently, both sides of the equation are constant and we denote the value by $-\lambda$. We only consider the left-hand side, given by
\begin{align*}
f^{\prime\prime}(x)=-\lambda\rho(x)f(x).
\end{align*}
By integration with respect to the Lebesgue measure we get
\begin{align*}
f^{\prime}(x)-f^{\prime}(0)=-\lambda\int_0^x f(y)\rho(y)dy,
\end{align*}
which can be written as
\begin{align*}
f^{\prime}(x)-f^{\prime}(0)=-\lambda\int_0^x f(y)d\mu(y),
\end{align*}
where $\rho$ is the density of the measure $\mu$. By applying the definition of $\Delta_{\mu}$,
\begin{align*}
\Delta_{\mu} f = -\lambda f,
\end{align*}
which yields 
\begin{align} \label{wave_equation_intro}
\Delta_{\mu}u=\frac{\partial^2 u}{\partial t^2  },
\end{align}
as a generalization of wave equation \eqref{wave_equation_intro}, since this equation does not involve the density $\rho$. Consequently, we can use it to formulate the problem for measure which possess no density, in particular for fractal measures on $[0,1]$.

 We are interested in the case where $\mu$ is a self-similar measure on a Cantor-like set. More precisely,
let $N\geq 2$ and $\{S_1,...,S_N\}$  be a finite family of affine contractions on $[0,1]$, i.e.
\begin{align*}
S_i: [0,1]\to[0,1], ~ S_i(x)=r_ix+b_i, ~0<r_i<1, ~ 0\leq b_i\leq 1-r_i, ~ i=1,...,N,
\end{align*} 
where  $S_1(0)=0<S_1(1)\leq S_2(0)<S_2(1)\leq ... <S_N(1)=1$.
Further, let  $\mu_1,...,\mu_N$, i.e. $\mu_1,...,\mu_N\in(0,1)$ weights and $\sum_{i=1}^N\mu_i=1$. It is known from \cite{HF} that a unique non-empty compact set $F\subseteq [0,1]$ exists such that 
\begin{align}
F=\bigcup_{i=1}^M S_i(F)\label{definition_set}
\end{align}
and a unique Borel probabiliy measure $\mu$ such that 
\begin{align}
\mu(A)=\sum_{i=1}^N\mu_i\mu\left(S_i^{-1}(A)\right)\label{self_similarity}
\end{align}
for any Borel set $A\subseteq [0,1]$. Further, it holds $\supp\mu=F$. We call the set $F$ Cantor-like set. 
Wave equations where $\mu$ is defined by an IFS with overlaps and has full support were investigated in \cite{CNT}. 

\par 
By adding a random external force, more precicely, a space-thime white noise $\xi$ on $L^2([0,1],\mu)$, we are concerned with the hyerpbolic stochastic PDE
\begin{align}
\begin{split}
\frac{\partial^2}{\partial t^2}u(t,x)&=\Delta_{\mu}^b u(t,x)+f(t,u(t,x))\xi(t,x),\label{spde_wave_intro}\\
u(0,x)&=u_0(x),\\
\frac{\partial}{\partial t}u(0,x)&=u_1(x),
\end{split}
\end{align}
where $b\in\{N,D\}$ determines the boundary condition.
It is known (see \cite{WI}) that the stochastic wave equation defined by the classical one-dimensional weak Laplacian $\Delta_{\lambda^1}$ has a unique mild solution which is, some regularity conditions provided, essentially $\frac{1}{2}$-Hölder continuous in space and in time. Here, essentially $\alpha$-Hölder continuous means Hölder continuous for every exponent strictly less than $\alpha$. In two space dimensions it turns out that the mild solution is a distribution, no function (see \cite{WI}). Hambly and Yang \cite{HYC} addressed the questions regarding these properties in the setting of a p.c.f. self-similiar set (in the sense of \cite{KA}) with Hausdorff dimension between one and two. However, the damped wave equation in their paper is a system of first-order SPDEs. According to the knowledge of the author, there are no results regarding these properties in case of second-order Walsh SPDEs defined by a fractal Laplacian. The Krein-Feller operator can be interpreted as a fractal Laplacian on sets with dimension less or equal one.  \par 
We prepare the formulation of the main theorem by stating the following regularity conditions, where $\gamma$ is the spectral exponent of $\Delta_{\mu}^b$ and $\delta\coloneqq \max_{1\leq i\leq N}\frac{\log \mu_i}{\log\left((\mu_ir_i)^{\gamma}\right)}$ is an indicator for the skewness of $\mu$.

\begin{hyp}\label{hypo} 
\begin{enumerate}[label=(\roman*)]
\item \label{hypoi} $\delta+1<\frac{1}{\gamma}$
\item  $u_0\in\mathcal{D}\left(\Delta_{\mu}^b\right)$ 
, $u_1\in\mathcal{D}\left((-\Delta_{\mu}^b)^{\frac{1}{2}}\right)$  
\item \label{hypoiii} There exists $q\geq 2$ such that $f$ is predictable and satisfies the following Lipschitz and linear growth conditions: There exists $L>0$ and a real predictable process $M:\Omega\times[0,T]\to\mathbb{R}$ with $\sup_{s\in[0,T]}\lVert M(s)\rVert_{L^q(\Omega)}<\infty$ such that for all $(w,t,x,y)\in\Omega\times[0,T]\times\mathbb{R}$
\begin{align*}
|f(\omega,t,x)-f(\omega,t,y)|&\leq L|x-y|,\\ |f(\omega,t,x)|&\leq M(w,t)+L|x|.
\end{align*} 
\end{enumerate}
\end{hyp}

Note that Condition \ref{hypoi} is satisfied if $\mu$ is the $d_H$-dimensional Hausdorff meausre on $F$, where $d_H$ is the Hausdorff dimension of $F$, with the exception of $\lambda^1$ on $[0,1]$.\\

We formulate the main result of the present paper, where $d_H$ is the Hausdorff dimension of $F$ and $\nu_{\min}\coloneqq \min_{1\leq i\leq N}\frac{\mu_i}{r_i^{d_H}}$.

\begin{thm}\label{main_theorem}
Let $T\geq 0$ and assume Condition \ref{hypo} with $q\geq 2$. Then, there exists a unique mild solution $\{u(t,x): 0\leq t\leq T, ~0\leq x\leq 1\}$ to SPDE \eqref{spde_wave_intro}. Furthermore, there exists a version of this solution such that the following holds:
\begin{enumerate}[label=(\roman*)]
\item If $q>2$ and $t\in[0,T]$, $u(t,\cdot)$ is a.s. essentially $\frac{1}{2}-\frac{1}{q}$-Hölder continuous on $[0,1]$.
\item If $q>\left(d_H+1+\frac{\log(\nu_{\min})}{\log(r_{\max)}}\right)^{-1}$ and $x\in[0,1]$, $u(\cdot,x)$ is a.s. essentially $\frac{1}{d_H+1+\frac{\log(\nu_{\min})}{\log(r_{\max)}}}-\frac{1}{q}$-Hölder continuous on $[0,T]$.
\end{enumerate}
\end{thm}
\begin{figure}[t]
\centering
\includegraphics[scale=0.35]{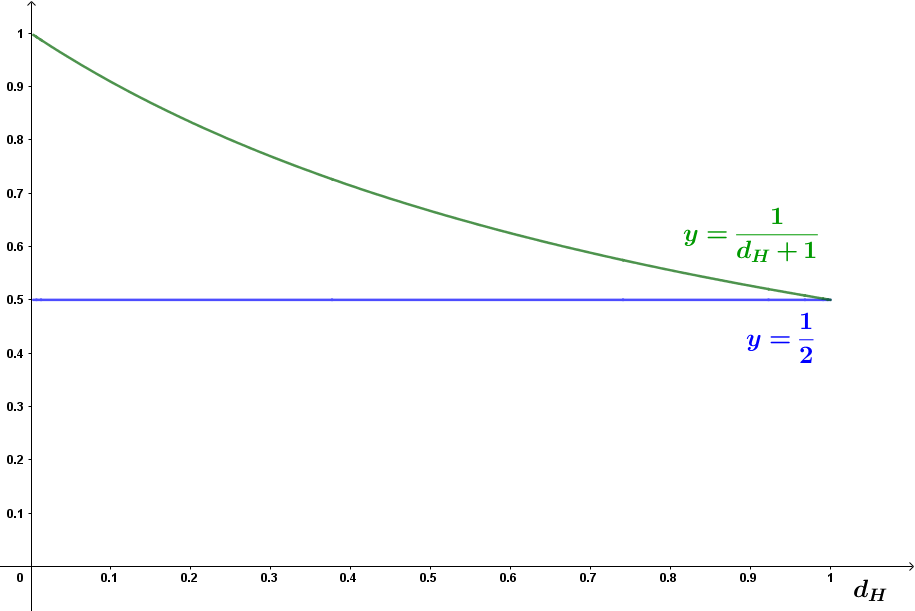}
\caption{Hölder exponent graphs}
\label{hoelder_plot}
\end{figure}

If $\mu$ is chosen as the natural measure, we have $\nu_{\min}=0$ and thus an increasing temporal Hölder exponent as the Hausdorff dimension of the considered Cantor-like set decreases. In particular, if $q$ can be chosen arbitrarily large, we obtain $\frac{1}{2}$ as ess. spatial and $\frac{1}{d_H+1}$ as ess. temporal Hölder exponent. This is visualized in Figure \ref{hoelder_plot}.\\

In preparation for proving the main results, we will have a closer look on the wave propagator of $\Delta_{\mu}^b$, defined by
\begin{align*}
P_D(t,x,y)=\sum_{k\geq 1}\frac{\sin\left(\sqrt{\lambda_k^D}t\right)}{\sqrt{\lambda_k^D}}\varphi_k^D(x)\varphi_k^D(y)
\end{align*}
and
\begin{align*}
P_N(t,x,y)=t+\sum_{k\geq 2}\frac{\sin\left(\sqrt{\lambda_k^N}t\right)}{\sqrt{\lambda_k^N}}\varphi_k^N(x)\varphi_k^N(y),
\end{align*}
respectively.
Here, $\lambda_k^b, k\geq 1$ are the eigenvalues and $\varphi_k^b, ~ k\geq 1$ the $L^2(\mu)$-normed eigenfunctions of the Neumann- (or Dirichlet- resp.) Krein-Feller operator $\Delta_{\mu}^b$. In order to investigate this object, we establish an improved estimate on the uniform norm of $\varphi_k^b$ since the known estimate, which grows exponentially in $k$ (see \cite[Lemma 4.1.6]{AE}), is to rough for our purposes. Particularly, we prove that a constant $C_2>0$ exists such that for all $k\in\mathbb{N}$
\begin{align*}
\left\lVert \varphi_k^b\right\rVert_{\infty}\leq C_2k^{\frac{\delta}{2}}.
\end{align*} 
A comparable result is known for the eigenfunction of p.c.f. Laplacians (see \cite[Theorem 4.5.4]{KA}). Afterwards, we approximate the wave propagator by proving that for $x\in F$, $t\in[0,T]$
\begin{align*}
\int_0^1\left(\left\langle P_b(t,\cdot,y),f_n^x\right\rangle - P_b(t,x,y)\right)^2d\mu(y)\to 0
\end{align*}

 as $n\to\infty,$ where the sequence $\left(f_n^x\right)_{n\in\mathbb{N}}$ approximates the Delta functional of $x$. Then, we show that the resulting approximating mild solutions have the desired continuity and that the regularity is preserved upon taking the limit. 
 Next to these continuity properties, we investigate the intermittency of mild solutions to \eqref{spde_wave_intro}. Roughly speaking, an intermittent process develops increasingly high peaks on small space-intervals when the time parameter  increases. This is a phenomenon of the mild solution to SPDEs that has found much attention in the last years (see, among many others, \cite{BCS}, \cite{HHS}, \cite{KAS} \cite{KKI} for parabolic and \cite{CJI}, \cite{DMI}, \cite{CKO} for hyperbolic SPDEs).  We call a mild solution $u$ weakly intermittent on $[0,1]$ if the upper moment Lyapunov exponents, which is the function $\bar \gamma$ defined by
\begin{align*}
\bar\gamma(p,x)\coloneqq \limsup_{t\to\infty}\frac{1}{t}\log\mathbb{E}\left[u(t,x)^p\right], ~~ p\in(0,\infty), x\in[0,1]
\end{align*}
satisfies
\begin{align*}
\bar\gamma(2,x)>0,~~~ \bar\gamma(p,x)<\infty, ~~ ~p\in [2,\infty), x\in[0,1]
\end{align*}
(see \cite[Definition 7.5]{KAS}).
We prove this under some conditions on $g$. \\
 
This paper is structured as follows. In Section \ref{section_2} we give definitions related to Krein-Feller operators and Cantor-like sets, we recall results concerning the spectral asymptotics and establish the explained estimate on the uniform norm of eigenfunctions. Furthermore, we develop a method to approximate the resolvent density and introduce the wave propagator.
Section \ref{section_3}  is dedicated to the analysis of SPDE \eqref{spde_wave_intro}, including the proofs of existence, uniqueness and Hölder continuity properties of the mild solution as well as the investigation of weak intermittency.
\par

\section{Preliminaries and Preparing Estimates}\label{section_2}
\subsection{Definition of Krein-Feller Operators on Cantor-like Sets}\label{Krein-Feller Operators and Cantor-like Sets}
First, we recall the definition and some analytical properties of the operator $\Delta_{\mu}^b$, where $b\in\{N,D\}$ and $\mu$ is a self-similar measure on a Cantor-like set according to the definition in Section \ref{section_1}. \par 
We denote the support of the measure $\mu$ and thus the Cantor-like set by $F$. If $[0,1]\setminus F\neq\emptyset$, $[0,1]\setminus F$ is open in $\mathbb{R}$ and can be written as
\begin{align}
[0,1]\setminus F = \bigcup_{i=1}^{\infty} (a_i,b_i)\label{offene_mengen} 
\end{align}
with $0<a_i<b_i<1$, $a_i,b_i\in[0,1]$ for $i\geq 1$. We define 
\begin{align*}
\mathcal{D}_{\lambda^1}^1\coloneqq\left\lbrace f:[0,1]\to\mathbb{R}: \text{there exists } f^{\prime}\in \mathcal{L}^2([0,1],\lambda^1): f(x)=f(0)+\int_0^x f^{\prime}(y)d\lambda^1(y),~ x\in[0,1]  \right\rbrace
\end{align*}
and $H^1\left([0,1],\lambda^1\right)$ as the space of all $\mathcal{H}\coloneqq L^2([0,1],\mu)$-equivalence classes having a $\mathcal{D}_{\lambda^1}^1-$represen\-tative. If $\mu=\lambda^1$ on $[0,1]$, this definition is equivalent to the definition of the Sobolev space $W_2^1$.

$H^1\left([0,1],\lambda^1\right)$ is the domain of the non-negative symmetric bilinear form $\mE$ on $\mathcal{H}$ defined by
\begin{align*}
\mathcal{E}(u,v)=\int_0^1 u'(x)v'(x)dx,~~~ u,v\in \mathcal{F}\coloneqq H^1\left([0,1],\lambda^1\right).
\end{align*}
It is known (see \cite[Theorem 4.1]{FD}) that $\left(\mE,\mathcal{F}\right)$ defines a Dirichlet form on $\mathcal{H}$. Hence, there exists an associated non-negative, self-adjoint operator $\Delta_{\mu}^N$ on $\mathcal{H}$ with $\mathcal{F}=\mathcal{D}\left(\left(-\Delta_{\mu}^N\right)^{\frac{1}{2}}\right)$ such that
\begin{align*}
\langle -\Delta_{\mu}^N u,v\rangle_{\mu}&=\mE(u,v), ~~u\in\mathcal{D}\left(\Delta_{\mu}^N\right),v\in \mathcal{F}
\end{align*}
and it holds 
\[\mathcal{D}\left(\Delta_{\mu}^N\right)=\left\{f\in \mathcal{H}: f \text{ has a representative } \bar f \text{ with } \bar f\in \mathcal{D}_{\mu}^2 \text{ and } \bar f'(0)=\bar f'(1)=0\right\}.\] $\Delta_{\mu}^N$ is called Neumann Krein-Feller operator w.r.t. $\mu$. Furthermore, let $\mathcal{F}_0\coloneqq H^1_0\left([0,1],\lambda^1\right)$ be the space of all $\mathcal{H}$-equivalence classes that have a $\mathcal{D}_{\lambda^1}^1-$representative $f$ such that $f(0)=f(1)=0.$
The bilinear form defined by
\begin{align*}
\mathcal{E}(u,v)=\int_0^1 u'(x)v'(x)dx,~~~ u,v\in \mathcal{F}_0
\end{align*}
is a Dirichlet form, too (see \cite[Theorem 4.1]{FD}).
Again, there exists an associated non-negative, self-adjoint operator $\Delta_{\mu}^D$ on $\mathcal{H}$ with $\mathcal{F}_0=\mathcal{D}\left(\left(-\Delta_{\mu}^D\right)^{\frac{1}{2}}\right)$ such that
\begin{align*}
\langle -\Delta_{\mu}^D u,v\rangle_{\mu}&=\mE(u,v), ~~ u\in \left(\Delta_{\mu}^D\right),~v\in \mathcal{F}_0
\end{align*}
and it holds 
\[\mathcal{D}\left(\Delta_{\mu}^D\right)=\left\{f\in \mathcal{H}: f \text{ has a representative } \bar f \text{ with } \bar f\in \mathcal{D}_{\mu}^2 \text{ and } \bar f(0)=\bar f(1)=0\right\}.\] $\Delta_{\mu}^D$ is called Dirichlet Krein-Feller operator w.r.t. $\mu$.\par

A concept to describe Cantor-like sets is given by the so-called word or code space. Let $I\coloneqq \{ 1,...,N\}$, $\mathbb{W}_n= I^n$ be the set of all sequences $\omega$ of length $|\omega|=n$, $\mathbb{W}^*\coloneqq \cup_{n\in\mathbb{N}} I^n,$ 
the set of all finite sequences and $\mathbb{W}\coloneqq I^{\infty}$ the set of all infinite sequences $\theta=\theta_1\theta_2\theta_3...$ with $\theta_i\in I$ for $i\in\mathbb{N}$. Then, $I$ is called alphabet and $\mathbb{W},~ \mathbb{W}^*,~ \mathbb{W}^n: ~ n\in\mathbb{N}$ are called word spaces. We define an ordering on $\mathbb{W}$ by denoting two words $\omega$ and $\sigma$ as equal if $\omega_i=\sigma_i$ for all $i\in\mathbb{N}$ and otherwise, we write
$\omega<\sigma :\Leftrightarrow \sigma_k<\omega_k$ or
$\omega>\sigma :\Leftrightarrow \sigma_k>\omega_k$, where $k\coloneqq\inf\{ n\in\mathbb{N}:\sigma_n\neq\omega_n\}$. In addition to an ordering we define a metric on the word space by the map $d:\mathbb{W}\times \mathbb{W}\rightarrow \mathbb{R},~ d(\omega,\sigma)=N^{-k}$ with $k$ defined as before. It is known (see e.g. \cite[Theorem 2.1]{BF})  
that for every $x\in [0,1]$ the map
\begin{align*}
\pi_x: \mathbb{W}\rightarrow F, ~~\sigma\mapsto\lim_{n \to\infty} S_{\sigma_1}\circ S_{\sigma_2}\circ...\circ S_{\sigma_n}(x) 
\end{align*}
is well-defined, continuous, surjective and independent of $x\in [0,1]$, which means $\pi_x(\sigma)=\pi_y(\sigma)$ for all $x,y\in [0,1],~ \sigma\in \mathbb{W}$. Therefore, for every $x\in[0,1]$ and every $y\in F$ there exists, at least, one element of $\mathbb{W}$ which is by $\pi_x$ associated to $y$.

\subsection{Spectral Theory of Krein-Feller Operators}
Let $b\in\{N,D\}$
and let $\mu$ be a self-similar measure on a Cantor-like set according to the given conditions.
Further, let $\gamma$ be the spectral exponent of $-\Delta_{\mu}^b$, that is the unique solution of
\begin{align}
\sum_{i=1}^N(\mu_ir_i)^{\gamma}=1. \label{spectral_exponent}
\end{align}

It is known from \cite[Proposition 6.3, Lemma 6.7, Corollary 6.9]{FA} that there exists an orthonormal basis $\lbrace \varphi_k^b: k\in\mathbb{N}\rbrace$ of $L_2([0,1],\mu)$ consisting of $L_2([0,1],\mu)$-normed eigenfunctions of $-\Delta_{\mu}^b$ and that for the related ascending ordered eigenvalues $\{\lambda^{b}_i:i\in\mathbb{N}\}$ it holds $0\leq\lambda_1^b\leq\lambda_2^b\leq...,$ where $\lambda_1^D>0$. 
Furthermore, by \cite{FAF} there exist constants $C_0,C_1>0$ such that for $k\geq 2$ 
\begin{align}
C_0k^{\frac{1}{\gamma}}\leq \lambda_k^b\leq C_1k^{\frac{1}{\gamma}}. \label{eigenwertabsch}
\end{align}

Next, we consider the uniform norm of an eigenfunction $\left\lVert \varphi_k^b\right\rVert_{\infty}$ for $k\geq 1$, where the situation is more complicated. The only estimate, established in \cite[Section 2]{FLZ} and \cite[Lemma 3.6]{AM}, is easy to derive and grows exponentially in $k$, which is  far to rough for later following heat kernel estimates. In the following proposition we establish a better estimate, where we do not use the explicit representation of the eigenfunctions as in \cite{AM}, but the ideas from \cite[Theorem 4.5.4]{KA} for a uniform norm estimate for Laplacians on p.c.f. fractals.

\begin{thm}\label{eigenfunction_estimate_theorem}
Let $\delta\coloneqq \max_{1\leq i\leq N}\frac{\log \mu_i}{\log\left((\mu_ir_i)^{\gamma}\right)}$. Then, there exists a constant $\bar C_2>0$ such that for all $k\in\mathbb{N}$ 
\begin{align*}
\left\lVert\varphi_k^b\right\rVert_{\infty}\leq \bar{C_2} \left(\lambda_k^b\right)^{\frac{\gamma}{2}\delta}.
\end{align*}
\end{thm}
Hereby, $\left\lVert f\right\Vert_{\infty}\coloneqq \text{ess sup}_{x\in[0,1]}|f(x)|$. This is an estimate for the essential supremum, but it also holds for the supremum of the representative in $\mathcal{D}^2_{\mu}$, since this representative is continuous on $[0,1]$ and linear on $(a_i,b_i),$ $i\in\mathbb{N}$,  Inequality \eqref{eigenwertabsch} implies with $C_2\coloneqq C_1^{\frac{\delta}{2}}\bar{C_2}$
\begin{align}
\left\lVert\varphi_k^b\right\rVert_{\infty}\leq C_2 k^{\frac{\delta}{2}}.\label{eigenfunction_estimate}
\end{align}
To verify Theorem \ref{eigenfunction_estimate_theorem}, we follow closely the proof of \cite[4.5.4]{KA}. First, we need a few preparations. Thereby, $\mE(u) \coloneqq \mE(u,u)$.
\begin{lem}\label{dirichletform_abs_1}
There exists a constant $c_0>0$ such that for all $u\in\mathcal{F}_0$
\begin{align*}
\left\lVert u\right\rVert^2_{\mu}\leq c_0\mathcal{E}(u).
\end{align*}
\end{lem}
\begin{proof}
It holds $\lambda_1^D>0$ and therefore (compare \cite[Theorem 1.3]{EES}) \[\mathcal{E}(u)\geq\lambda_1^D\left\lVert u\right\rVert_{\mu}^2, ~~ u\in\mathcal{F}_0.\] 
\end{proof}
\begin{lem}\label{2norm_absch}
There is a constant $c_1>0$ such that for all $u\in \mathcal{F}$
\begin{align*}
\left\lVert u\right\rVert^2_{\mu}\leq c_1\left(\mathcal{E}(u)+\left\lVert u\right\rVert_1^2\right),
\end{align*}
where $\left\lVert f\right\rVert_{1}\coloneqq \int_0^1 |f(x)|d\mu(x)$.
\end{lem}
\begin{proof}
Let $u\in\mathcal{F}$ and $u_0$ be the unique harmonic function with $u_0(0)=u(0)$ and $u_0(1)=u(1)$, that is $u_0(x)\coloneqq u(0)(1-x)+u(1)x.$ We have $(u-u_0)(0)=(u-u_0)(1)=0$ and thus $u-u_0\in\mathcal{F}_0$. Since the space of harmonic functions on $[0,1]$ with two boundary conditions is two-dimensional, there exists $c_1'>0 $ such that for all harmonic functions $u_0$
\begin{align*}
\left\lVert u_0\right\rVert_{\mu}\leq c_1'\left\lVert u_0\right\rVert_1
\end{align*}
and since $\mu$ is a probability measure we have for all $u\in\mathcal{F}$
\begin{align*}
\left\lVert u \right\rVert_1\leq \left\lVert u\right\rVert_{\mu}.
\end{align*}
Furthermore,
\begin{align*}
\mE(u-u_0)&=\mE(u)-2\mE(u,u_0)+\mE(u_0)\\
&=\mE(u)-2\int_0^1u'(x)(u(1)-u(0))dx+(u(1)-u(0))^2\\
&=\mE(u)-2(u(1)-u(0))^2+(u(1)-u(0))^2\\
&=\mE(u)-(u(1)-u(0))^2
\end{align*}
and thus
\begin{align*}
\mE(u-u_0)\leq\mE(u).
\end{align*}
By Lemma \ref{dirichletform_abs_1} and the above calculations,
\begin{align*}
\left\lVert u \right\rVert_{\mu}&\leq \left\lVert u_0 \right\rVert_{\mu}+\left\lVert u-u_0\right\rVert_{\mu}\\
&\leq c_1'\left\lVert u_0\right\rVert_1+\sqrt{c_0\mathcal{E}(u-u_0,u-u_0)}\\
&\leq c_1'(\left\lVert u\right\rVert_1+\left\lVert u-u_0\right\rVert_1)+\sqrt{c_0\mathcal{E}(u-u_0,u-u_0)}\\
&\leq c_1'(\left\lVert u\right\rVert_1+\left\lVert u-u_0\right\rVert_{\mu})+\sqrt{c_0\mathcal{E}(u-u_0,u-u_0)}\\
&\leq c_1'\left\lVert u\right\rVert_1+c_1'\sqrt{c_0\mathcal{E}(u-u_0,u-u_0)}+\sqrt{c_0\mathcal{E}(u-u_0,u-u_0)}\\
&\leq 2c_0^{\frac{1}{2}}c_1'(\left\lVert u\right\rVert_1+\sqrt{\mathcal{E}(u)}).
\end{align*}
The assertion follows from the fact that for positive numbers $a,b,c$ with  $a\leq b+c$ it holds~ $a^2\leq 2(b^2+c^2)$.
\end{proof}
Moreover, we need scaling properties for $\mu$ and $\mE$. Preliminary, we introduce the notion of a partition (see \cite[Definition 1.3.9]{KA}).
\begin{defi}
For $\omega\in\mathbb{W}^*$ let
$
\Sigma_{\omega}\coloneqq \{\sigma=\sigma_1\sigma_2...\in\mathbb{W}: \sigma_i=\omega_i \text{ for all } 1\leq i\leq |\omega|\}.
$
 A finite subset $\Lambda\subset \mathbb{W}^*$ is called partition if it holds $\Sigma_{\omega}\cap\Sigma_{\sigma}=\emptyset$ for $\omega\neq\sigma\in\Lambda$ and $\mathbb{W}=\bigcup_{\omega\in\Lambda}\Sigma_{\omega}$.
\end{defi}
We introduce some notation for the following lemma. Let $w\in\mathbb{W}^*$. For a function $f$  we define $f_{\omega}\coloneqq f_{\omega_1}\circ f_{\omega_2}\circ ...\circ f_{\omega_{|\omega|}}$. Analogously, the pushforward measure $\mu\circ f_{\omega_1}\circ f_{\omega_2}\circ ...\circ f_{\omega_{|\omega|}}$ is denoted by  $\mu\circ f_{\omega}$.
\begin{lem}\label{scaling_prop}
Let $\Lambda$ be a partition. It holds
\begin{itemize}
\item[(i)] 
$\begin{aligned}[t]
\mu=\sum_{\omega\in\Lambda}\mu_{\omega}(\mu\circ S_{\omega}^{-1}) 
\end{aligned}$,
\item[(ii)] 
$\begin{aligned}[t]
\sum_{\omega\in\Lambda}r_{\omega}^{-1}\mE(u\circ S_{\omega})
&\leq \mE(u) \text{ for all } u\in\mathcal{F}.
\end{aligned}$
\end{itemize}
\end{lem}
We skip the proof of this lemma since it works by standard arguments, as in \cite[Section 3.2.1]{AE}.
\begin{proof}[Proof of Theorem \ref{eigenfunction_estimate_theorem}] Let $u\in\mathcal{F}$ be fixed. Then,
\begin{align}
\left\lVert u\right\rVert_{\mu}^2&=\int_0^1u^2(x)d\mu(x)\notag\\
&=\sum_{\omega\in\Lambda}\mu_w\int_0^1 u^2(x)d\mu\circ S_{\omega}^{-1}(x)\label{eigenfunction_estimate_main_proof_1}\\
&=\sum_{\omega\in\Lambda}\mu_w\int_0^1 u(S_{\omega}(x))^2d\mu(x)\notag\\
&\leq c_1\sum_{\omega\in\Lambda}\mu_w\left(\mE(u\circ S_{\omega})+\left\lVert u\circ S_{\omega}\right\rVert_1^2\right)\label{eigenfunction_estimate_main_proof_2}\\
&\leq c_1\left(\max_{\omega\in\Lambda}\{\mu_{\omega}r_{\omega}\}\sum_{\omega\in\Lambda}r_{\omega}^{-1}\mE(u\circ S_{\omega})+\sum_{\omega\in\Lambda}\mu_w^{-1}\left(\mu_{\omega}\int_0^1|u\circ S_{\omega}|d\mu\right)^2\right)\\
&\leq c_1\left(\max_{\omega\in\Lambda}\{\mu_{\omega}r_{\omega}\}\mE(u)+\min_{\omega\in\Lambda}\{\mu_{\omega}^{-1}\}\left\lVert u\right\rVert_1^2\right).\label{eigenfunction_theorem_proof_1}
\end{align}
Hereby, equation \eqref{eigenfunction_estimate_main_proof_1} follows from Lemma \ref{scaling_prop}(i), inequality \eqref{eigenfunction_estimate_main_proof_2} from Lemma \ref{2norm_absch} and inequality \eqref{eigenfunction_theorem_proof_1} from  Lemma \ref{scaling_prop}(ii).
Now, let $\nu_i\coloneqq (\mu_ir_i)^{\gamma}, ~ i=1,...,N$. By  \eqref{spectral_exponent} it holds $\sum_{i=1}^N\nu_i=1$.
Let $\lambda\in(0,1)$ and the partition $\Lambda_{\lambda}$ defined by
\[\Lambda_{\lambda}=\lbrace \omega\in\mathbb{W}^*: \nu_{\omega_1}\cdots\nu_{\omega_{|\omega|-1}}>\lambda\geq\nu_{\omega}\rbrace.\]
By definition of  $\Lambda_{\lambda}$ we have for $\omega\in\Lambda_{\lambda}$
$\nu_{\omega}^{\frac{1}{\gamma}} =$
$\mu_{\omega}r_{\omega}\leq\lambda^{\frac{1}{\gamma}}$ and from that
$\max_{\omega\in\Lambda_{\lambda}}(\mu_{\omega}r_{\omega})\leq\lambda^{\frac{1}{\gamma}}$. 
Furthermore, it is known from \cite[Proposition 4.5.2]{KA} that there exists $C_2'>0$, such that 
$
\min_{\omega\in\Lambda_{\lambda}}\mu_{\omega}\geq C_2'\lambda^{\delta}$, from which it follows $\left(\min_{\omega\in\Lambda_{\lambda}}\mu_{\omega}\right)^{-1}\leq\frac{1}{C_2'} \lambda^{-\delta}$. 
This and \eqref{eigenfunction_theorem_proof_1} yield to the existence of a constant $C_2''>0$ such that for all $\lambda\in(0,1)$, $u\in\mathcal{F}$
\begin{align*}
\left\lVert u\right\rVert_{\mu}^2\leq C_2''\left(\lambda^{\frac{1}{\gamma}}\mE(u)+\lambda^{-\delta}\left\lVert u\right\rVert_1^2\right).
\end{align*}
Let $\theta\coloneqq 2\gamma\delta$. We assume that $\mE(u)>\left\lVert u\right\rVert_1^2$ and choose $\lambda\in(0,1)$ such that $\lambda^{\frac{1}{\gamma}+\delta}=\frac{\left\lVert u\right\rVert_1^2}{\mE(u)}$. It follows
\begin{align}
\left\lVert u\right\rVert_{\mu}^2\leq 2C_2''\lambda^{-\delta}\left\lVert u\right\rVert_1^2 \label{eigenfunction_estimate_main_proof_5}
\end{align}
and therefore
\begin{align}
\left\lVert u\right\rVert_{\mu}^{\frac{4}{\theta}}&\leq (2C_2'')^{\frac{2}{\theta}}\lambda^{-\frac{2\delta}{\theta}}\left\lVert u\right\rVert_1^{\frac{4}{\theta}}\notag\\
&=(2C_2'')^{\frac{2}{\theta}}\lambda^{-\frac{1}{\gamma}}\left\lVert u\right\rVert_1^{\frac{4}{\theta}}. \label{eigenfunction_estimate_main_proof_6}
\end{align}
By combining \eqref{eigenfunction_estimate_main_proof_5} and \eqref{eigenfunction_estimate_main_proof_6}  we get
\begin{align*}
\left\lVert u\right\rVert_{\mu}^{2+\frac{4}{\theta}}&\leq (2C_2'')^{1+\frac{2}{\theta}}\lambda^{-\delta -\frac{1}{\gamma}}\left\lVert u\right\rVert_1^{2+\frac{4}{\theta}}\\
&=(2C_2'')^{1+\frac{2}{\theta}}\left\lVert u \right\rVert_1^{\frac{4}{\theta}}\mE(u).
\end{align*}
If it holds $\mE(u)\leq\left\lVert u\right\rVert_1^2$, from Lemma \ref{2norm_absch} it follows
\begin{align*}
\left\lVert u\right\rVert_{\mu}^2\leq 2c_1\left\lVert u\right\rVert_1^2
\end{align*}
and thus
\begin{align*}
\left\lVert u\right\rVert_{\mu}^{2+\frac{4}{\theta}}\leq \left(2c_1\right)^{1+\frac{2}{\theta}}\left\lVert u\right\rVert_1^2\left\lVert u\right\rVert_1^{\frac{4}{\theta}}.
\end{align*}
All in all, there is a $C_2'''>0$ such that for all $u\in\mathcal{F}$ the Nash-type inequality 
\begin{align}
\left\lVert u\right\rVert_{\mu}^{2+\frac{4}{\theta}}\leq C_2'''\left(\mE(u)+\left\lVert u\right\rVert_{\mu}^2\right)\left\lVert u\right\rVert_1^{\frac{4}{\theta}}\label{nash}
\end{align}
is fulfilled.
Let $\psi: L^2([0,1],\mu)\to L^2(F,\mu)$, $f\to \left.f\right|_{F}$ and
$ \widetilde\Delta_{\mu}^N:~ \psi\left(\mathcal{D}\left(\Delta_{\mu}^N\right)\right)\to L^2(F,\mu)$,
  $u\to\psi\circ\Delta_{\mu}^N\circ \psi^{-1}u$. Then, $\widetilde\Delta_{\mu}^N$ is self-adjoint, has eigenvalues $\lambda_k^N$ with eigenfunctions $\psi\circ\varphi_k^N$ for $k\in\mathbb{N}$ and the Dirichlet form $\widetilde\mE\left(\widetilde u,\widetilde v\right)\coloneqq\mE(\psi^{-1}\widetilde u,\psi^{-1}\widetilde v)$, $\widetilde u,\widetilde v\in\widetilde{\mathcal{F}}\coloneqq\psi(\mathcal{F})$ is associated (see Appendix \ref{appendix_restriction_supp}).

Then, for all $\widetilde u\in\widetilde{\mathcal{F}}$ the Nash-type inequality
\begin{align}
\left\lVert\widetilde u\right\rVert_{\mu}^{2+\frac{4}{\theta}}\leq C_2'''\left(\widetilde{\mE}\left(\widetilde u\right)+\left\lVert\widetilde u\right\rVert_{\mu}^2\right)\left\lVert\widetilde u\right\rVert_1^{\frac{4}{\theta}}
\end{align}
is satisfied. Since it holds $\mu(O)>0$ for all open sets $O\subseteq F$, we can apply \cite[Proposition B.3.7]{KA} to get the existence of $C_2''''>0$ such that for all $k\in\mathbb{N}$
\begin{align}
\left\lVert\widetilde T_t^N\widetilde \varphi_k^N\right\rVert_{\infty}\leq C_2''''t^{-\frac{\theta}{4}}, \label{eigenfunction_estimate_semigroup}
\end{align}
where $\left(\widetilde T_{t}^N\right)_{t\geq 0}$ is the strongly continuous semigroup associated to $ \widetilde\Delta_{\mu}^N.$
With $\widetilde T_t^N\widetilde\varphi_k^N=e^{-\lambda_k^N t}\widetilde\varphi_k^N$ for $t\geq 0$ (see \cite [Corollary B.2.7]{KA}), $t\coloneqq \frac{1}{\lambda_k^N}$ and $\bar C_2\coloneqq C_2''''e$
we obtain for all $k\in\mathbb{N}$
\begin{align*}
\left\lVert \widetilde \varphi_k^N\right\rVert_{\infty}\leq \bar C_2\lambda_k^{\frac{\gamma\delta}{2}},
\end{align*}
from which the assertion follows for $b=N$ since $\varphi_k^b$ is linear on the intervals in $F^c$. In case of $b=D$ the proof works analogously since $\mathcal{F}_0\subseteq \mathcal{F}$.
\end{proof}

\subsection{Properties of the Resolvent Operator}\label{resolvent_chapter}

For $\lambda>0$ and $b\in\{N,D\}$ let $\rho_{\lambda}^b$ be the resolvent density of $\Delta_{\mu}^b$. That is, with $R^{\lambda}_{b}\coloneqq (\lambda-\Delta_{\mu}^b)^{-1}$ it holds
\begin{align*}
R^{\lambda}_{b}f(x)=\int_0^1\rho_{\lambda}^b(x,y)f(y)d\mu(y), ~~~ f\in \mH.
\end{align*}
Such a mapping exists and is given by (compare \cite[Theorem 6.1]{FA})
\begin{align*}
\rho_{\lambda}^N(x,y) = \rho_{\lambda}^N(y,x) = \left(B^{\lambda}_N\right)^{-1} g_{1,N}^{\lambda}(x)g_{2,N}^{\lambda}(y), ~~ x,y\in[0,1], x\leq y, \\
\rho_{\lambda}^D(x,y) = \rho_{\lambda}^D(y,x) = \left(B^{\lambda}_D\right)^{-1}g_{1,D}^{\lambda}(x)g_{2,D}^{\lambda}(y), ~~ x,y\in[0,1], x\leq y, 
\end{align*}
where $B^{\lambda}_N,B^{\lambda}_D$ are non-vanishing constants and  the mappings $g_{1,N}^{\lambda}, g_{2,N}^{\lambda}, g_{1,D}^{\lambda}, g_{2,D}^{\lambda}$ are eigenfunctions of $\Delta_{\mu}$ with appropriate boundary conditions (see \cite[Remark 5.2]{FA}). We prove that the resolvent density is Lipschitz. 

\begin{satz}\label{resolvent_density_lipschitz}
Let $\lambda>0$. Then, for every $\lambda>0$ there exists a constant $L_{\lambda}\geq 0$ such that
\begin{align*}
\left| \rho_{\lambda}^b(x,y)-\rho_{\lambda}^b(x,z)\right|\leq L_{\lambda}|y-z|, ~~~ x,y,z\in[0,1].
\end{align*}
\end{satz}
\begin{proof}
Let $b\in\{N,D\}$. We denote the maximum of the Lipschitz constants of the functions  $g_{1,b}^{\lambda}, g_{2,b}^{\lambda}$ (according to the amount) by $L_{\lambda}'$ and $\max\left\{\left\lVert g_{1,b}^{\lambda}\right\rVert_{\infty},\left\lVert g_{2,b}^{\lambda}\right\rVert_{\infty}\right\}$ by $L_{\lambda}''$. Now, let $x\in[0,1].$ For $y,z\in[x,1]$ we have
\begin{align*}
\left| \rho_{\lambda}^b(x,y)-\rho_{\lambda}^b(x,z)\right|= \left| \left(B^{\lambda}_b\right)^{-1}\right|\left|g_{1,b}^{\lambda}(x)\left(g_{2,b}^{\lambda}(y)-g_{2,b}^{\lambda}(z)\right)\right|\leq \left| \left(B^{\lambda}_b\right)^{-1}\right| L_{\lambda}^{'}L_{\lambda}^{''}|y-z|.
\end{align*}
From the symmetry we get the same for $y,z\in[0,x]$. For $0\leq z\leq x\leq y\leq 1$ we have
\begin{align*}
\left| \rho_{\lambda}^b(x,y)-\rho_{\lambda}^b(x,z)\right|&=\left| \left(B^{\lambda}_b\right)^{-1}\right|\left|g_{1,b}^{\lambda}(x)g_{2,b}^{\lambda}(y)-g_{1,b}^{\lambda}(z)g_{2,b}^{\lambda}(x)\right|\\
&\leq 
\left| \left(B^{\lambda}_b\right)^{-1}\right|\Bigg(\left|g_{1,b}^{\lambda}(x)g_{2,b}^{\lambda}(y)-g_{1,b}^{\lambda}(x)g_{2,b}^{\lambda}(x)\right|\\
&~~~~~~~~~~~~~~~~~~+
\left|g_{1,b}^{\lambda}(x)g_{2,b}^{\lambda}(x)-g_{1,b}^{\lambda}(z)g_{2,b}^{\lambda}(x)\right| \Bigg) \\
&\leq \left| \left(B^{\lambda}_b\right)^{-1}\right|L_{\lambda}^{'}L_{\lambda}^{''}\left(|y-x|+|x-z|\right) \\ &= \left| \left(B^{\lambda}_b\right)^{-1}\right|L_{\lambda}^{'}L_{\lambda}^{''}|y-z|
\end{align*}
and, again, the symmetry implies the same for $0\leq y\leq x\leq z\leq 1$.
\end{proof}

\subsection{Approximation of the Resolvent Density}
We develop a method to approximate the delta functional on Cantor-like sets, in particular to approximate the just introduced resolvent density, which will then again (dann wiederrum) be used to approximate point evaluations of heat kernels.\par 
For $n\geq 1$ let $\Lambda_n$ be the partition of the word space $\mathbb{W}$ be defined by
\begin{align*}
\Lambda_n=\{\omega=\omega_1...\omega_m\in \mathbb{W}^*:
 r_{\omega_1}\cdots r_{\omega_{m-1}}>r_{\max}^{n}\geq r_{\omega}\},
\end{align*}
where $r_{\max}\coloneqq\max_{i=1,...,N}r_i$.
Moreover, let $\nu_i=\frac{\mu_i}{r_i^{d_H}}, 1\leq i\leq N$, where $d_H$ is the Hausdorff dimension of $F$. Further, for $\omega\in\mathbb{W}$ we denote $S_{\omega}(F)$ by $F_{\omega}$.
\begin{lem}\label{partitionslemma}
It holds for $n\in\mathbb{N}$:
\begin{enumerate}[label=(\roman*)]
\item \label{partitionslemma_1} $|\Lambda_n|<\infty$ and $\bigcup_{\omega\in \Lambda_n}F_{\omega}=F.$
\item \label{partitionslemma_2} For $\omega\in \Lambda_n$ there exists a subset $\Lambda'\subseteq\Lambda_{n+1}$ such that $F_{\omega}=\bigcup_{\nu\in\Lambda'}F_{\nu}.$
\item \label{partitionslemma_3} For $\omega,\nu\in\Lambda_n$, $\omega\neq\nu$ it holds $|F_{\omega}\cap F_{\nu}|\leq 1.$
\item \label{partitionslemma_estimate} For $\omega\in\Lambda_n$ it holds $\mu(F_{\omega})>r_{\max}^{nd_H}r_{\min}^{d_H}\nu_{\min}^n$.
\item \label{partitionslemma_5} For $w\in\mathbb{W}^*$ there exists $n\in\mathbb{N}$ such that $w\in\Lambda_n$. Consequently, for all $m\geq n$ there exists $\Lambda'_m\subseteq \Lambda_m$ such that $F_w=\cup_{\nu\in\Lambda'_m}F_{\nu}$.
\end{enumerate}
\end{lem}

If the measure $\mu$ is chosen as $\mu_i=r_i^{d_H}$ and thus $\nu_i = 1,~ i=1,...,N$, we get an estimate similar to \cite[Lemma 3.5(iv)]{HYC}. Note that these ideas can be used to generalize the corresponding results in \cite{HYC}.
\begin{proof}
\begin{enumerate}[label=(\roman*)]
\item\label{partitionslemma_i} The first claim is obvious. For the second we note that $\cup_{w\in\mathbb{W}}F_w=F$ and that $\cup_{w\in\Lambda_n}\Sigma_w=\mathbb{W}$ and thus $\cup_{v\in\Sigma_w, w\in\Lambda_n}F_v=F.$ It remains to show that $F_w=\cup_{v\in\Sigma_w}F_v$ for $w\in\Lambda_n$. This follows from applying $f_w$ to both sides of the equation $\cup_{\nu\in\mathbb{W}}F_{\nu}=F$.
\item\label{partitionslemma_ii} Let $\omega\in\Lambda_n$. We know from part \ref{partitionslemma_i} that $F_w=\cup_{\nu\in \Sigma_w}F_{\nu}$. If $r_w\leq r_{\max}^{n+1}$, the assertion follows since we can choose $\Lambda'=\{w\}$. Now, we assume $r_w> r_{\max}^{n+1}$. Then it holds for $i=1,...,N$ $r_{\omega}r_i\leq r_{\max}^{n+1}$, since $r_{\omega}\leq r_{\max}^n$. It follows $wi\in\Lambda_{n+1}$ for $i=1,...,N$. We get the result by using this and applying $f_w$ to both sides of equality \eqref{definition_set}.
\item Since $\omega\neq\nu$, there exists an $m\leq\min\{|\omega|,|\nu|\}$ such that $\omega_m\neq\nu_m$. From $|\text{Im}(f_i)\cap \text{Im}(f_j)|\leq 1$ for $1\leq i\neq j\leq N$ it follows
\begin{align*}
|f_{\omega_m}\circ f_{\omega_{m+1}}\circ\cdots\circ f_{\omega_{|\omega|}}(F)\cap
f_{\nu_m}\circ f_{\nu_{m+1}}\circ\cdots\circ f_{\nu_{|\nu|}}(F)|\leq 1.
\end{align*}
The assertion follows by composing with the respective maps $f_{\omega_{m-1}},...,f_{\omega_{1}},f_{\nu_{m-1}},f_{\nu_{1}}$ and using the injectivity if $\omega_i=\nu_i$ and the disjointness of the images, except at most one point if $\omega_i\neq\nu_i$  for $i<m$.\par 
\item Let $\omega\in\Lambda_n$ and $m\coloneqq|\omega|.$ By definition of $\Lambda_n$ it holds $r_{\omega_1}\cdots r_{\omega_{m-1}}>r_{\max}^{n}$ and therefore $r_{\omega}>r_{\max}^{n}r_{\min}$. By using that,
\begin{align*}
\mu_{\omega}&=r_{\omega_1}^{d_H}\frac{\mu_{\omega_1}}{r_{\omega_1}^{d_H}}\cdots r_{\omega_m}^{d_H}\frac{\mu_{\omega_m}}{r_{\omega_m}^{d_H}}\\
&\geq r_{\omega}^{d_H}\nu_{\min}^m\\
&> r_{\max}^{nd_H}r_{\min}^{d_H}\nu_{\min}^m\\
&\geq r_{\max}^{nd_H}r_{\min}^{d_H}\nu_{\min}^n.
\end{align*}
The last inequality follows from $m\leq n$ and $\nu_{\min}\leq 1$.
\item Let $w=w_1....w_m\in\mathbb{W}^*$. Choose $n\in\mathbb{N}$ such that $r_w\leq r_{\max}^n$ and $r_w> r_{\max}^{n+1}$. From $r_{w_1}...r_{w_{m-1}}r_{\max}>r_{w_1}...r_{w_{m}}$ it follows 
\[r_{w_1}...r_{w_{m-1}}>r_{w_1}...r_{w_{m}}r_{\max}^{-1}>r_{\max}^{n+1}r_{\max}^{-1}=r_{\max}^{n}.\]
Therefore, we can find an $n\in\mathbb{N}$ such that $w\in\Lambda_{n}$. For the second part, we can argue as in \ref{partitionslemma_ii} with induction. 
\end{enumerate}
\end{proof}

We introduce a sequence of functions approximating the Delta functional. Hereby, we use the notation of \cite{HYC}.  We prepare this definition by defining the $n$-neighbourhood of $x\in F$ for $n\in\mathbb{N}$ by
\begin{align*}
D^0_n(x)\coloneqq \bigcup\{F_w:w\in\Lambda_n,~ x\in F_w\}.
\end{align*}
Note that $D_n^0(x)$ consists of at least one element of $\{F_w, w\in\Lambda_n\}$, which follows from Lemma \ref{partitionslemma}\ref{partitionslemma_1}, and of at most two elements since pairs of these elements intersect in at most one point. From the latter and the definition of $\Lambda_n$ it follows 
\begin{align}
|D_n^0(x)|\leq 2r_{\max}^n.\label{Dn0_diameter}
\end{align}
 With that, we can define the approximating functions for $x\in F$ and $n\geq 1$ by
\begin{align*}
f_n^x=\mu(D_n^0(x))^{-1}\mathds{1}_{D_n^0(x)}.
\end{align*}
From Lemma \ref{partitionslemma}\ref{partitionslemma_estimate} it follows
\begin{align}
||f_n^x||^2_{\mu}=\mu(D_n^0(x))^{-1}\leq r_{\max}^{-nd_H}r_{\min}^{-d_H}\nu_{\min}^{-n}.\label{f_n^x_norm_esimate}
\end{align}
We deduce the following result.
\begin{lem}\label{approximation_result}
Let $x\in F$. It holds $\lim_{n\to\infty}\langle f_n^x,g\rangle_{\mu}=g(x)$ for any continuous $g\in\mH$.
\end{lem}
\begin{proof}
For $n\in\mathbb{N}$ and $\omega\in \Lambda_n$ it holds $|F_{\omega}|\leq r_{\max}^n$ since $|F|\leq 1$ and $r_{\omega}\leq r_{\max}^n$. Therefore, it holds $|y-x|\leq r_{\max}^n$ for $x,y\in D^0_n(x)$. Now, let $x\in F$ and $\varepsilon>0$. Since $g$ is continuous in $x$, there exists $\delta>0$ such that $|g(x)-g(y)|<\varepsilon$ for $y\in[0,1]$ with $|y-x|<\delta$. Choose $n\in\mathbb{N}$ such that $r_{\max}^n<\delta.$ Then, it follows
\begin{align*}
|\langle f_n^x,g\rangle_{\mu}-g(x)|&=\frac{1}{\mu(D_n^0(x))}\Big|\int_{D_n^0(x)}g(y)d\mu(y)-g(x)\Big|\\
&\leq\frac{1}{\mu(D_n^0(x))}\int_{D_n^0(x)}|g(y)-g(x)|d\mu(y)\\
&\leq \frac{1}{\mu(D_n^0(x))}\mu(D_n^0(x))\cdot\varepsilon = \varepsilon.
\end{align*}
\end{proof}

\begin{lem}\label{resolvent_estimate}
Let $x_1,x_2\in F$ and $m,n\geq 1.$ Then,
\begin{align*}
\left|\int_0^1\int_0^1\rho_1^b(y,z)f_m^{x_1}(y)f_n^{x_2}(z)d\mu(y)d\mu(z)-\rho_1^b(x_1,x_2)\right|\leq2 L_1(r_{\max}^n+r_{\max}^m),
\end{align*}
where $L_1$ denotes the Lipschitz constant of $\rho_1^b$. 
\end{lem}
\begin{proof}
By using the Lipschitz continuity of $\rho_1^b$ and \eqref{Dn0_diameter},
\begin{align*}
&\left|\int_0^1\int_0^1\left(\rho_1^b(y,z)-\rho_1^b(x_1,x_2)\right)f_m^{x_1}(y)f_n^{x_2}(z)d\mu(y)d\mu(z)\right|\\
&\leq \int_0^1\int_0^1\left|\rho_1^b(y,z)-\rho_1^b(x_1,z)\big|+\big|\rho_1^b(x_1,z)-\rho_1^b(x_1,x_2)\right|f_m^{x_1}(y)f_n^{x_2}(z)d\mu(y)d\mu(z)\\
&= \frac{1}{\mu(D_m^0(x_1))\mu(D_n^0(x_2))}\Bigg(\int_{D_m^0(x_1)}\int_{D_n^0(x_2)}\big|\rho_1^b(y,z)-\rho_1^b(x_1,z)\big|\\ &~~~~~~~~~~~~~~~~~~~~~~~~~~~~~~~~~~~~~~~~~~~~~~~~~~~~~~+\big|\rho_1^b(x_1,z)-\rho_1^b(x_1,x_2)\big|
d\mu(y)d\mu(z)\Bigg)\\
&\leq \frac{1}{\mu(D_m^0(x_1))\mu(D_n^0(x_2))}\int_{D_m^0(x_1)}\int_{D_n^0(x_2)}2L_1 \left(r_{\max}^m+r_{\max}^n\right)d\mu(y)d\mu(z)\\
&= 2L_1 \left(r_{\max}^m+r_{\max}^n\right).
\end{align*}
\end{proof}
\par

\subsection{Wave Equations Defined by Fractal Laplacians}
We introduce the notion of a wave propagator in this section which will be used to define the concept of a mild solution to \eqref{spde_wave_intro}.
 Let $T>0$, $b\in\{N,D\}$, $u_0=\sum_{k\geq 1}u_{0,k}^b\varphi_k^b\in\mathcal{D}\left(\Delta_{\mu}^b\right)$ and $u_1=\sum_{k\geq 1}u_{1,k}^b\varphi_k^b\in\mathcal{D}\left(\left(-\Delta_{\mu}^b\right)^{\frac{1}{2}}\right)$, that is $\sum_{k\geq 1}\left(\lambda_k^b\right)^2\left(u_{0,k}^b\right)^2<\infty$ and $\sum_{k\geq 1}\lambda_k^b\left(u_{1,k}^b\right)^2<\infty$.  $\Delta_{\mu}^b$ is a self-adjoint, dissipative operator on $\mathcal{H}$. Hence, it is well-known (compare, e.g., \cite{SA})  that the wave equation 
\begin{align}\label{wave_equation_1}
\begin{cases}
\frac{\partial^2}{\partial t^2}u(t,x)&=-\Delta_{\mu}^N u(t,x),\\
u(0,x)&=u_0(x),\\
\frac{\partial u(0,x)}{\partial t}&=u_1(x)
\end{cases}
\end{align}
on $[0,T]\times[0,1]$ has a unique solution given by
\begin{align*}
u(t,x)=\int_0^1 P_b(t,x,y)u_1(y)d\mu(y)+\frac{\partial}{\partial t}\int_0^1 P_b(t,x,y)u_0(y)d\mu(y), ~~(t,x)\in[0,T]\times[0,1],
\end{align*}
where $P_b$ is the so-called wave propagator (see \cite[Chapter 5]{DSF}), defined by
\begin{align*}
P_D(t,x,y)=\sum_{k\geq 1}\frac{\sin\left(\sqrt{\lambda_k^D}t\right)}{\sqrt{\lambda_k^D}}\varphi_k^D(x)\varphi_k^D(y)
\end{align*}
and
\begin{align*}
P_N(t,x,y)=t+\sum_{k\geq 2}\frac{\sin\left(\sqrt{\lambda_k^N}t\right)}{\sqrt{\lambda_k^N}}\varphi_k^N(x)\varphi_k^N(y),
\end{align*}
respectively. Note that,
\begin{align*}
\int_0^1P_N(t,x,y)u_0(y)d\mu(y)&=tu_{1,k}^D+\sum_{k\geq 2}\frac{\sin\left(\sqrt{\lambda_k^N}t\right)}{\sqrt{\lambda_k^N}}u_{1,k}^N\varphi_k^N(x),\\
\frac{\partial}{\partial t}\int_0^1P_N(t,x,y)u_0(y)d\mu(y)&=u_{1,k}^N+\sum_{k\geq 2}\cos\left(\sqrt{\lambda_k^N}t\right)u_{0,k}^N\varphi_k^N(x),
\\
\int_0^1P_D(t,x,y)u_1(y)d\mu(y)&=\sum_{k\geq 1}\frac{\sin\left(\sqrt{\lambda_k^D}t\right)}{\sqrt{\lambda_k^D}}u_{1,k}^D\varphi_k^D(x),\\
\frac{\partial}{\partial t}\int_0^1P_D(t,x,y)u_0(y)d\mu(y)&=\sum_{k\geq 1}\cos\left(\sqrt{\lambda_k^D}t\right)u_{0,k}^D\varphi_k^D(x).
\end{align*}
Recall that $\gamma$ is the spectral exponent of $-\Delta_{\mu}^b$ (compare \eqref{spectral_exponent}) and $\delta\coloneqq \max_{1\leq i\leq N}\frac{\log \mu_i}{\log\left((\mu_ir_i)^{\gamma}\right)}$.
\begin{lem}[Properties of the wave propagator] \label{propagator_prop}
\item Let $\delta+1<\frac{1}{\gamma}$. Then, there exists a constant $C_3>0$ such that
\begin{align*}
\sup_{t\in[0,\infty)}\sup_{x\in[0,1]}\lVert P_b(t,x,\cdot)\rVert_{\mu}<C_3.
\end{align*}
\end{lem}
\begin{proof}
By Proposition \ref{eigenfunction_estimate_theorem} and \eqref{eigenwertabsch}, for all $t\geq 0$ and $x\in[0,1]$
\begin{align*}
\left\lVert \sum_{k\geq 2}\frac{\sin\sqrt{\lambda_k}t}{\sqrt{\lambda_k}}\varphi_k(x)\varphi_k(\cdot)\right\rVert_{\mu}^2=&
t^2+\sum_{k\geq 2}\frac{\sin^2\sqrt{\lambda_k}t}{\lambda_k}\varphi_k^2(x)\\
&\leq T^2+C_2\sum_{k\in\mathbb{N}}\frac{1}{\lambda_k}k^{\delta}\\
&\leq T^2+ C_1C_2 \sum_{k\in\mathbb{N}} k^{\delta-\frac{1}{\gamma}},
\end{align*}
which is finite independently of $x\in[0,1]$ and $t\in[0,T]$ due to the assumption.
\end{proof}

\section{Analysis of Stochastic Wave Equations}\label{section_3}
\subsection{Preliminaries}

Let $(\Omega,\mathcal{F},\mathbb{F}\coloneqq (\mathcal{F}_t)_{t\geq 0},\mathbb{P})$ be a filtered probability space statisfying the usual conditions. We consider the stochastic PDE 
\begin{align}
\begin{split}
\frac{\partial^2}{\partial t^2}u(t,x)&=\Delta_{\mu}^b u(t,x)+f(t,u(t,x))\xi(t,x)\label{spde1}\\
u(0,x)&=u_0(x),\\
\frac{\partial}{\partial t}u(0,x)&=u_1(x)
\end{split}
\end{align}
for $(t,x)\in[0,T]\times [0,1]$, where $T>0$, $b\in\{N,D\}$ determines the boundary conditions, $u_0,u_1:[0,1]\to\mathbb{R}$, $f:\Omega\times[0,T]\times[0,1]\to\mathbb{R}$ and $\xi$ is a $\mathbb{F}$-space-time white noise on $([0,1],\mu)$, that is a mean-zero set-indexed Gaussian process on $\mathcal{B}\left([0,T]\times[0,1]\right)$ such that $\mathbb{E}\left[\xi(A)\xi(B)\right]=\left|A\cap B\right|$ (compare \cite[Chapter 1]{WI}). Moreover, let for a time interval $I\subseteq[0,T]$ and a space interval $J\subseteq[0,\infty)$ $\mathcal{P}_{I,J}$ be the $\sigma$-algebra generated by simple functions on $\Omega\times I\times J$, where a simple function on $\Omega\times I \times J$ is defined as a finite sum of functions $h:\Omega\times I\times J\to\mathbb{R}$ of the form
\begin{align*}
h(\omega,t,x)=X(\omega)\mathds{1}_{(a,b]}(t)\mathds{1}_{B}(x), ~~ (\omega,t,x)\in \Omega\times I\times J
\end{align*}
with $X$ bounded and $\mathcal{F}_a$-measurable, $a,b\in I,a<b$ and $B\in\mathcal{B}(J)$.
\begin{defi}
Let $q\geq 2, T>0$ be fixed. Let $S_{q,T}$ be the space of $[0,T]\times[0,1]$-indexed processes $v$ being predictable (i.e. measurable from $\mathcal{P_{[0,T],[0,1]}}$ to $\mathcal{B}(\mathbb{R})$ and satisfying
\begin{align*}
||v||_{q,T}\coloneqq \sup_{t\in[0,T]}\sup_{x\in[0,1]}\left(\mathbb{E}|v(t,x)|^q\right)^{\frac{1}{q}}<\infty.
\end{align*}
Furthermore, define $\mathcal{S}_{q,T}$ as the space of equivalence classes of processes in $\mathcal{S}_{q,T}$, where two processes $v_1,v_2$ are equivalent if $v_1(t,x)=v_2(t,x)$ almost surely for all $(t,x)\in[0,T]\times[0,1]$. 
\end{defi}
Note that $\mathcal{S}_{q}$ and $\mathcal{S}_{q,T}$ are Banach spaces. The proof works by using standard arguments, so we skip it here. \par

We define the concept of a solution to \eqref{spde1} which we observe in this paper.

\begin{defi}
A \textbf{mild solution} to the SPDE \eqref{spde1} is defined as a predictable $[0,T]\times[0,1]$-indexed process such that for every $(t,x)\in[0,T]\times[0,1]$ it holds almost surely
\begin{align}
\begin{split}
u(t,x)=\frac{\partial}{\partial t}\int_0^1&P_b(t,x,y)u_0(y)d\mu(y)+\int_0^1P_b(t,x,y)u_1(y)d\mu(y)\\+&\int_0^t\int_0^1P_b(t-s,x,y)f(s,u(s,y))\xi(s,y)d\mu(y)ds,\label{mild_solution_wave}
\end{split}
\end{align}
for $(t,x)\in[0,T]\times[0,1]$, where the last term is a stochastic integral in the sense of \cite[Chapter 2]{WI}. 
\end{defi}

In this chapter, we assume that Condition \ref{hypo} is satisfied for a given $q\geq 2$. Note that since $u_0\in\mathcal{D}\left(-\Delta^{\mu}\right)$, it holds  \[\left|u_{0,k}\right|<C_4 k^{\frac{1}{\gamma}}, k\geq 1, \] where $C_4\coloneqq \left\lVert\left(-\Delta^{\mu}\right)^{\frac{1}{2}}u_0\right\rVert_{\mu}$. Analogously, since $u_1\in\mathcal{D}\left((-\Delta_{\mu}^b)^{\frac{1}{2}}\right)$, we have \[\left|u_{1,k}\right|<C_5 k^{\frac{1}{2\gamma}}, ~ k\geq 1, \]
where $C_5\coloneqq\left\lVert\left(-\Delta^{\mu}\right)^{\frac{1}{2}}u_1\right\rVert_{\mu}$.

\subsection{Existence, Uniqueness and Continuity}
Let $b\in\{N,D\}$. In this section, we prove continuity properties of $v_i, i=1,2,3$, which are defined as follows for $v_0\in\mathcal{S}_{q,T}$ :
\begin{align}
v_1(t,x)&\coloneqq \int_0^t\int_0^1P_b(t-s,x,y)f(s,v_0(s,y))\xi(s,y)d\mu(y)ds,\label{v1}\\
v_2(t,x)&\coloneqq \int_0^1P_b(t,x,y)u_1(y)d\mu(y),\label{v2}\\
v_3(t,x)&\coloneqq \frac{\partial}{\partial t}\int_0^1 P_b(t,x,y)u_0(y)d\mu(y).\label{v3}
\end{align}

We need some preparing lemmas. The following lemma shows how to find upper estimates of functionals of the wave propagator by using the resolvent density.

\begin{lem}\label{estimate_resolvent_lemma}
For all $t\in(0,T]$ and $g\in\mathcal{H}$ it holds
\begin{align*}
\int_0^1\left(\int_0^1 P_b(t,x,y)g(y)d\mu(y)\right)^2d\mu(x)\leq 2t^2 \int_0^1 \rho_1^b(x,y)g(x)g(y)d\mu(x)d\mu(y).
\end{align*}
\begin{proof}
Let $b=N$ and $g=\sum_{k=0}^{\infty}g_k\varphi_k^N$. Then, since the sequence of eigenvalues is increasing,
\begin{align}
\int_0^1\left(\int_0^1 P_N(t-s,x,y)g(y)d\mu(y)\right)^2d\mu(x)&=\left\lVert tg_0+\sum_{k=2}^{\infty}\frac{\sin\left(\sqrt{\lambda_k^N}(t-s)\right)}{\sqrt{\lambda_k^N}}g_k\varphi_k^N \right\rVert^2_{\mu}\notag\\
&=  t^2g_0^2+\sum_{k=2}^{\infty} \frac{\sin^2\left(\sqrt{\lambda_k^N}(t-s)\right)}{\lambda_k^N}g_k^2\notag\\
&\leq t^2\left(g_0^2+\sum_{k=2}^{\infty} \frac{1}{\lambda_k^N}g_k^2\right) \notag\\
&\leq t^2 \frac{1+\lambda_2^N}{\lambda_2^N}\sum_{k=1}^{\infty} \frac{1}{1+\lambda_k^N}g_k^2\notag\\
&= t^2 \frac{1+\lambda_2^N}{\lambda_2^N}\left\langle g,\left(1-\Delta_{\mu}^N\right)^{-1}g\right\rangle_{\mu}.\label{spatial_estimate_resolvent}
\end{align}
By definition of the resolvent density it holds 
\[
\left(1-\Delta_{\mu}^N\right)^{-1}g = \int_0^1\rho_1^N(\cdot,y)g(y)d\mu(y).
\]
Plugging this and the fact that $\lambda_2^N>1$ (see, e.g.,  \cite[Section 3.3.1]{AE}) into \eqref{spatial_estimate_resolvent}, the assertion for $b=N$ follows. The case $b=D$ works similarly using $\lambda_1^D>1$  (see, e.g., \cite[Lemma 4.9]{MS}).
\end{proof}
\end{lem}

This leads to a useful approximation of $P_b(t,x,\cdot)$ for fixed $(t,x)\in [0,T]\times[0,1]$.

\begin{lem} \label{estimate_resolvent_lemma_2}
Let $x\in F$. Then, there exists a constant $C_6$ such that for all $t\in(0,\infty)$ and $n\in\mathbb{N}$
\begin{align*}
\int_0^1\left(\left\langle P_b(t,\cdot,y),f_n^x\right\rangle - P_b(t,x,y)\right)^2d\mu(y)\leq C_6t^2r_{\max}^{n}.
\end{align*}
\end{lem}
\begin{proof}
Let $b=N$. For any $(t,y)\in[0,\infty)\times[0,1]$ $P_N(t,\cdot,y)$ is an element of $\mathcal{H}$ and the inner product on $\mathcal{H}$ is continuous in each argument. It thus holds for any $g\in\mathcal{H}$
\begin{align}
\left\langle P_N(t,\cdot,y),g\right\rangle_{\mu}&=
\left\langle t+\sum_{k=2}^{\infty}\frac{\sin\left(\sqrt{\lambda_k^N}(t)\right)}{\sqrt{\lambda_k^N}}\varphi_k^N(y)\varphi_k^N,g\right\rangle_{\mu}\notag\\
&=\left\langle \lim_{m\to\infty}t+\sum_{k=2}^{m}\frac{\sin\left(\sqrt{\lambda_k^N}(t)\right)}{\sqrt{\lambda_k^N}}\varphi_k^N(y)\varphi_k^N,g\right\rangle_{\mu}\notag\\
&= t\left\langle 1,g\right\rangle_{\mu}+\sum_{k=2}^{\infty}\frac{\sin\left(\sqrt{\lambda_k^N}(t)\right)}{\sqrt{\lambda_k^N}}\varphi_k^N(y)\left\langle\varphi_k^N,g\right\rangle_{\mu}.\label{inner_product_cont}
\end{align}
By using this, Lemma \ref{approximation_result} and Fatou's Lemma,
\begin{align}
&\int_0^1\left( \left\langle P_N(t,\cdot,y),f_n^x\right\rangle - P_N(t,x,y)\right)^2d\mu(y)\notag \\ &=
\int_0^1 \left(t\left\langle 1,f_n^x\right\rangle_{\mu}+\sum_{k=2}^{\infty}\frac{\sin\left(\sqrt{\lambda_k^N}(t)\right)}{\sqrt{\lambda_k^N}}\varphi_k^N(y)\left\langle \varphi_k^N,f_n^x\right\rangle - P_N(t,x,y)\right)^2d\mu(y)\notag \\
&=\int_0^1 \left( t\left(\left\langle 1,f_n^x\right\rangle_{\mu}-1\right)+
\sum_{k=2}^{\infty}\frac{\sin\left(\sqrt{\lambda_k^N}(t)\right)}{\sqrt{\lambda_k^N}}\left[\left\langle \varphi_k^N,f_n^x\right\rangle - \varphi_k^N(x)\right]\varphi_k^N(y)\right)^2d\mu(y)\notag \\
&= \sum_{k=2}^{\infty}\frac{\sin^2\left(\sqrt{\lambda_k^N}(t)\right)}{\lambda_k^N}\left[\left\langle \varphi_k^N,f_n^x\right\rangle - \varphi_k^N(x)\right]^2\label{resolvent_lemma_2_proof} \\
&= \sum_{k=2}^{\infty}\frac{\sin^2\left(\sqrt{\lambda_k^N}(t)\right)}{\lambda_k^N}\left[\left\langle \varphi_k^N,f_n^x\right\rangle - \lim_{m\to\infty}\left\langle\varphi_k^N,f_m^x\right\rangle\right]^2\notag \\
&= \sum_{k=2}^{\infty}\lim_{m\to\infty}\frac{\sin^2\left(\sqrt{\lambda_k^N}(t)\right)}{\lambda_k^N}\left[\left\langle \varphi_k^N,f_n^x\right\rangle - \left\langle\varphi_k^N,f_m^x\right\rangle\right]^2\notag \\
&\leq \liminf_{m\to\infty}\sum_{k=2}^{\infty}\frac{\sin^2\left(\sqrt{\lambda_k^N}(t)\right)}{\lambda_k^N}\left[\left\langle \varphi_k^N,f_n^x\right\rangle - \left\langle\varphi_k^N,f_m^x\right\rangle\right]^2\notag \\
&=\liminf_{m\to\infty}\int_0^1\left(\int_0^1 P_N(t-s,y,z)\left(f_n^x(z)-f_m^x(z)\right)d\mu(z)\right)^2d\mu(y)\notag .
\end{align}
Note that $ \left\langle 1,f_n^x\right\rangle_{\mu}=\int_0^1 f_n^x(y)d\mu(y)=1$, what we have used in equation \eqref{resolvent_lemma_2_proof}. 
By Lemma \ref{estimate_resolvent_lemma},
\begin{align}
&\int_0^1\left(\int_0^1 P_N(t,y,z)(f_n^x(z)-f_m^x(z)d\mu(z))\right)^2d\mu(y)\notag \\ &\leq
2t^2\int_0^1\int_0^1 \rho_1^N(x,y)(f^x_n(y)-f^x_m(y))(f^x_n(z)-f^x_m(z))d\mu(y)d\mu(z)\notag\\
&= 2t^2\left|\int_0^1\int_0^1\rho_1^b(z,y)\left(f_m^x(z)f_m^x(y)-f_m^x(z)f_n^x(y)-f_n^x(z)f_m^x(y)+f_n^x(z)f_n^x(y)\right)d\mu(z)d\mu(y)\right|\notag\\
&= 2t^2\Big|\int_0^1\int_0^1\rho_1^b(z,y)f_m^x(z)f_m^x(y)-\rho_1^b(x,x)-\rho_1^b(z,y)f_m^x(z)f_n^x(y)+\rho_1^b(x,x)\notag\\&~~~-\rho_1^b(z,y)f_n^x(z)f_m^x(y)+\rho_1^b(x,x)+\rho_1^b(z,y)f_n^x(z)f_n^x(y)-\rho_1^b(x,x)d\mu(z)d\mu(y)\Big|\notag\\
&\leq 16L_1t^2(r_{\max}^m+r_{\max}^n)\notag,
\end{align}
where we have used Lemma \eqref{resolvent_estimate} in the last inequality.
We conclude
\begin{align*}
\int_0^1\left( \left\langle P_N(t-s,\cdot,y),f_n^x\right\rangle - P_N(t-s,x,y)\right)^2d\mu(y) &\leq \liminf_{m\to\infty} 16L_1t^2(r_{\max}^n+r_{\max}^m)\\
&= 16L_1t^2r_{\max}^n.
\end{align*}
The case $b=D$ works similarly.
\end{proof}

We need one more estimate to find upper bounds for point evaluations of the wave propagator.
\begin{lem}\label{summenlemma}
Let $a<0, b\geq 0$. Then, there exists a constant $C_{a,b}$ such that for all $t\in[0,\infty)$
\begin{align*}
\sum_{k\in\mathbb{N}} k^{a-1}\wedge tk^{b-1}\leq C_{a,b} t^{\frac{-a}{b-a}}.
\end{align*}
\end{lem}
\begin{proof}
\cite[Lemma 5.2]{HYE}
\end{proof}

We are now able to prove Hölder continuity propoerties of $v_i,~ i=1,2,3.$ We start with the deterministic ones. 

\begin{satz}\label{v_2_continuity}
Let $T>0$ be fixed. Then, there exists a constant $C_7>0$ such that for all  $i\in\{2,3\},$ $t\in[0,T], x,y\in[0,1]$ $v_i(t,x)$ is well-defined and it holds
\begin{align*}
|v_i(t,x)-v_i(t,y)| &\leq C_7|x-y|,\\
|v_i(s,x)-v_i(t,x)|&\leq C_7|s-t|^{(2-(2+\delta)\gamma)\wedge 1}.
\end{align*}
\end{satz}
\begin{proof}
First, we consider the spacial continuity, where we use ideas from \cite[Proposition 4.1]{HYT}. Let $t\in[0,T]$, $x,y\in F$. By Lemma \ref{resolvent_estimate}, \ref{estimate_resolvent_lemma} and \ref{estimate_resolvent_lemma_2},
\begin{align*}
&\left|\int_0^1P_b(t,x,z)u_1(z)d\mu(z)-\int_0^1P_b(t,y,z)u_1(z)d\mu(z)\right|^2 \\
&\leq \int_0^1\left(P_b(t,x,z)-P_b(t,y,z)\right)^2u_1^2(z)d\mu(z)\\
&\leq  \sup_{z\in[0,1]}\left(u_1^2(z)\right)\int_0^1\left(P_b(t,x,z)-P_b(t,y,z)\right)^2d\mu(z)\\
& =\sup_{z\in[0,1]}\left(u_1^2(z)\right)\lim_{n\to\infty}\int_0^1\left(\left\langle P_b(t,\cdot,z),f_n^x-f_n^y\right\rangle\right)^2d\mu(z) 
\end{align*} 
\begin{align*}
&\leq 2\sup_{z\in[0,1]}\left(u_1^2(z)\right)t^2\lim_{n\to\infty} \left|\int_0^1\int_0^1\rho_1^b(z_1,z_2)(f_n^x(z_1)-f_n^y(z_1))(f_n^x(z_2)-f_n^y(z_2)) d\mu(z_1)d\mu(z_2)\right|\\
&=2\sup_{z\in[0,1]}\left(u_1^2(z)\right)t^2\left|\rho_1^b(x,x)-2\rho_1^b(x,y)+\rho_1^b(y,y)\right|\\
&\leq 4L_1\sup_{z\in[0,1]}\left(u_1^2(z)\right)t^2|x-y|.
\end{align*} 
Recall that $[0,1]\setminus F = \bigcup_{i=1}^{\infty} (a_i,b_i)$ (see \eqref{offene_mengen}). Now, let $b=N$ and $x,y\in F^c$ such that there exists an $i\in\mathbb{N}$ with $(x,y)\in(a_i,b_i)$, where we assume $x<y$. Then, since $a_i,b_i\in F$, the previous calculation implies
\begin{align}
&\left|\int_0^1\left(P_N(t,x,z)-P_N(t,y,z)\right)^2u_1^2(z)d\mu(z)\right|\notag\\
&\leq \sup_{z\in[0,1]}\left(u_1^2(z)\right) \sum_{k=2}^{\infty}\frac{\sin^2\left(\sqrt{\lambda_k^N}t\right)}{\lambda_k^N}\left(\varphi_k^N(x)-\varphi_k^N(y)\right)^2\notag\\
&\leq \sup_{z\in[0,1]}\left(u_1^2(z)\right) \left(\frac{x-y}{b_i-a_i}\right)^2\sum_{k=2}^{\infty}\frac{\sin^2\left(\sqrt{\lambda_k^N}t\right)}{\lambda_k^N}\left(\varphi_k^N(b_i)-\varphi_k^N(a_i)\right)^2 \label{v2_spatial_1}\\
&\leq 4L_1\sup_{z\in[0,1]}\left(u_1^2(z)\right)t^2\left(\frac{x-y}{b_i-a_i}\right)^2|b_i-a_i|\notag\\
&\leq 4L_1\sup_{z\in[0,1]}\left(u_1^2(z)\right)t^2\frac{(x-y)^2}{b_i-a_i}\notag\\
&\leq 4L_1\sup_{z\in[0,1]}\left(u_1^2(z)\right)t^2|x-y|\notag,
\end{align}
where we have used that for $k\in\mathbb{N}$ $\varphi_k^N$ is linear on $(a_i,b_i)$, $i\in\mathbb{N}$ in \eqref{v2_spatial_1}.
The remaining cases for $x,y\in[0,1]$ follow by using the triangle inequality for the $\mathcal{H}$-norm. Since the Dirichlet case works similarly, we obtain for all $(x,y)\in[0,1]$
\begin{align*}
\left|\int_0^1P_b(t,x,z)u_1(z)d\mu(z)-\int_0^1P_b(t,y,z)u_1(z)d\mu(z)\right|
&\leq 3\cdot 2^{\frac{1}{2}}\sup_{z\in[0,1]}\left(|u_1(z)|\right)2^{\frac{1}{2}}t|x-y|^{\frac{1}{2}}.
\end{align*}
We turn to $v_3$ and define $\widetilde{u}_{0,k}^b=\sqrt{\lambda_k^b}u_{0,k}^b$, $k\in\mathbb{N}$ for $k\geq 2$. With that, 
\begin{align*}
\frac{\partial}{\partial t}\int_0^1P_N(t,x,y)u_0(y)d\mu(y)&=u_{0,1}^N+\sum_{k\geq 2}\cos\left(\sqrt{\lambda_k^N}t\right)u_{0,k}^N\varphi_k^N(x)\\
&= u_{0,1}^N+\sum_{k\geq 2}\frac{\cos\left(\sqrt{\lambda_k^N}t\right)}{\sqrt{\lambda_k^N}} \widetilde{u}_{0,k}^N\varphi_k^N(x)
\end{align*}
and can now argue similar to the proof for $v_2$ since $\sum_{k\geq 2}\widetilde{u}_{0,k}^N\varphi_k^N\in\mathcal{D}\left(\left(\Delta_{\mu}^b\right)^{\frac{1}{2}}\right)$ as $u_1$. Again, the proof works analogously for Dirichlet boundary conditions. \par
For the temporal continuity, let $s,t\in[0,T]$ with $s<t$ and $x\in[0,1].$ Then,
\begin{align*}
&\left|\int_0^1\left(P_b(t,x,y)-P_b(s,x,y)\right)u_1(y)d\mu(y)\right| \\
&\leq (t-s)|u_{1,0}|+\sum_{k=2}^{\infty}\left(\frac{\sin\left(\sqrt{\lambda_k}(t)\right)-\sin\left(\sqrt{\lambda_k}(s)\right)}{\sqrt{\lambda_k}}\right)|\varphi_k(x)u_{1,k}|\\
&\leq (t-s)|u_{1,0}|+\sum_{k=2}^{\infty}\left(\frac{2\wedge \left( \sqrt{\lambda_k}(t)-\sqrt{\lambda_k}(s)\right) }{\sqrt{\lambda_k}}\right)|\varphi_k(x)u_{1,k}|\\
&\leq 2TC_0^{-\frac{1}{2}}C_2C_5\sum_{k=2}^{\infty} \left(k^{\frac{\delta}{2}-\frac{1}{\gamma}}\wedge\left(|s-t|k^{\frac{\delta}{2}-\frac{1}{2\gamma}}\right)\right).
\end{align*}
Choose $a=\frac{\delta}{2}-\frac{1}{\gamma}+1$ and $b=\frac{\delta}{2}-\frac{1}{2\gamma}+1$ in Lemma \ref{summenlemma} to get
\begin{align*}
\left|\int_0^1\left(P_N(t,x,y)-P_N(s,x,y)\right)u_1(y)d\mu(y)\right|
&\leq 2C_0^{-\frac{1}{2}}C_2C_5C_{a,b}|s-t|^{\left(2-(2+\delta)\gamma\right)\wedge 1}.
\end{align*}
With similar methods,
\begin{align*}
\sum_{k=1}^{\infty}\left|\cos\left(\sqrt{\lambda_k^N}t\right)-\cos\left(\sqrt{\lambda_k^N}s\right)\right||\varphi_k^N(x)||u_{0,k}^N|
&\leq 2C_1^{\frac{1}{2}} C_2C_4\sum_{k\in\mathbb{N}}\left(1\wedge k^{\frac{1}{2\gamma}} (t-s) \right)k^{\frac{\delta}{2}}k^{-\frac{1}{\gamma}}.
\end{align*}
Again, choose $a=\frac{\delta}{2}-\frac{1}{\gamma}+1$ and $b=\frac{\delta}{2}-\frac{1}{2\gamma}+1$ to get
\begin{align*}
\sum_{k\in\mathbb{N}}\left|\cos\left(\sqrt{\lambda_k^N}t\right)-\cos\left(\sqrt{\lambda_k^N}s\right)\right||\varphi_k^N(x)||u_{0,k}|&\leq 2C_1^{\frac{1}{2}} C_2C_4C_{a,b}|t-s|^{\left(\frac{\frac{1}{\gamma}-1-\frac{\delta}{2}}{\frac{1}{2\gamma}}\right)\wedge 1}\\ &=2C_1^{\frac{1}{2}} C_2C_4|t-s|^{(2-(2+\delta)\gamma) \wedge 1}.
\end{align*}
The calculation for Dirichlet boundary conditions works similarly.
\end{proof}

\begin{satz} \label{stochastic_continuity}
Let $q\geq 2$ and $T>0$ be fixed. Then, there exists a constant $c_8>0$ such that for all $v_0\in \mathcal{S}_{q,T}$ $v_1$ is well-defined, predictable and it holds for all $t\in[0,T], x,y\in[0,1]$
\begin{align*}
\mathbb{E}\left(|v_1(t,x)-v_1(t,y)|^q\right)&\leq c_8\left(1+\lVert v_0\rVert_{q,T}^q\right)|x-y|^{\frac{q}{2}}\\
\mathbb{E}\left(|v_1(s,x)-v_1(t,x)|^q\right)&\leq c_8\left(1+\lVert v_0\rVert_{q,T}^q\right)|s-t|^{\frac{q}{d_H+1+\frac{\log(\nu_{\min})}{\log(r_{\max})}}}.
\end{align*}
\end{satz}
\begin{proof}
For fixed $x\in[0,1]$, $P(\cdot,x,\cdot)$ is measurable and deterministic and therefore predictable and $f$ and $v_0$ are predictable, according to the assumption. Hence, the integrand in \eqref{v1} is predictable. By Hypothesis \ref{hypoiii} we have
\begin{align}
|f(t,v_0(t,x))|\leq M(t)+L|v_0(t,x)|, ~~ (t,x)\in[0,T]\times[0,1]\label{fprop1}.
\end{align}
With that, for $t\in[0,T],x\in[0,1]$
\begin{align*}
&\mathbb{E}\left[\int_0^T\int_0^1 P_b^2(t-s,x,y)\left(f(s,v_0(s,y))\right)^2d\mu(y)ds   \right] \\
&\leq \sup_{s\in[0,T]}\left\lVert M(s)\right\rVert_{L^2(\Omega)}+L\left\lVert v_0\right\rVert_{q,T}\int_0^T\int_0^1 P_b^2(t-s,x,y)d\mu(y)ds \\
&\leq \sup_{s\in[0,T]}\left\lVert M(s)\right\rVert_{L^2(\Omega)}+LT\left\lVert v_0\right\rVert_{q,T}\sup_{(s,x)\in[0,T]\times[0,1]}\int_0^1 P_b^2(s,x,y)d\mu(y),
\end{align*}
which is finite, independently of $x$ and $t$, due to Lemma \ref{propagator_prop}.
Consequently, $v_1$ is well-defined for $t,x\in[0,T]\times[0,1]$. We now prove the spatial estimate for $v_1$. For that, let $t\in[0,T],x,y\in[0,1]$ be fixed. Then, there exists $C_q>0$ such that
\begin{align}
\mathbb{E}&\left(|v_1(t,x)-v_1(t,y)|^q\right)\notag\\&=
\mathbb{E}\left(\left|\int_0^t\int_0^1\left(P_b(t-s,x,z)-P_b(t-s,y,z)\right)f(s,v_0(s,y))\xi(s,y)d\mu(z)ds \right|^q\right)\notag\\
&\leq C_q\left(\mathbb{E}\left(\left|\int_0^t\int_0^1\left(P_b(t-s,x,z)-P_b(t-s,y,z)\right)^2f(s,v_0(s,y))^2d\mu(z)ds \right|^{\frac{q}{2}}\right)\right)^{\frac{2}{q}\frac{q}{2}}\label{v1spartial1}\\
&\leq C_q\left|\int_0^t\int_0^1\left|\left(P_b(t-s,x,z)-P_b(t-s,y,z)\right)^q\mathbb{E}\left(f(s,v_0(s,y))^q\right)\right|^{\frac{2}{q}}d\mu(z)ds \right|^{\frac{q}{2}}\label{v1spartial2}\\
&= C_q\left|\int_0^t\int_0^1\left(P_b(t-s,x,z)-P_b(t-s,y,z)\right)^2\left|\mathbb{E}\left(f(s,v_0(s,y))^q\right)\right|^{\frac{2}{q}}d\mu(z)ds \right|^{\frac{q}{2}}\\
&\leq  2^{q-1}C_q\left(\lVert M\rVert_{q,T}^q+L^q\lVert v_0\rVert_{q,T}^q\right)\left|\int_0^t\int_0^1\left(P_b(t-s,x,z)-P_b(t-s,y,z)\right)^2d\mu(z)ds \right|^{\frac{q}{2}}\label{v1spartial3},
\end{align}
where we have used the Burkholder-Davis-Gundy inequality (see e.g. \cite[Theorem B.1]{KAS}) in \eqref{v1spartial1}, which can be used since the considered stochastic integral is a square-integrable martingale (see \cite[Theorem 2.5]{WI}, Minkowski's integral inequality in \eqref{v1spartial2} and the relation 
\begin{align}
\mathbb{E}\left(f(s,v_0(s,y))^q\right)\leq \mathbb{E}\left( M(s)+L|v_0(s,y)|\right)^q\leq 2^{q-1}\left(\mathbb{E}\left(M(s)^q\right)+L^q\mathbb{E}\left(\left|v_0(s,y)\right|^q\right)\right),\label{festimate}
\end{align}
which follows from \eqref{fprop1}, in \eqref{v1spartial3}. 
We proceed by estimating the integral term in \eqref{v1spartial3}, whereby we first treat the case $x,y\in F$. 
Analogously to the proof of Proposition \ref{v_2_continuity}, we calculate
\begin{align*}
&\int_0^t\int_0^1\left(P_b(t-s,x,z)-P_b(t-s,y,z)\right)^2d\mu(z)ds \\
&=\int_0^t\int_0^1\left(P_b(s,x,z)-P_b(s,y,z)\right)^2d\mu(z)ds \\
&\leq \int_0^t 4L_1s^2|x-y|ds\\
&\leq 4L_1\frac{t^3}{3} |x-y|
\end{align*}
Now, let $b=N$ and $x,y\in F^c$ such that there exists an $i\in\mathbb{N}$ with $(x,y)\in(a_i,b_i)$, where we assume $x<y$. Again, we can follow the proof of Proposition \ref{v_2_continuity} and get
\begin{align}
&\int_0^t\int_0^1\left(P_N(t,x,z)-P_N(t,y,z)\right)^2d\mu(z)ds\notag\\
&\leq \int_0^t 4L_1s^2|x-y|ds\notag\\
&\leq 4L_1\frac{t^3}{3} |x-y|\notag
\end{align}
The remaining cases for $x,y\in[0,1]$ follow by using the triangle inequality for the norm $L^2([0,T]\times[0,1],\lambda^1\times\mu)$, whereby this works analogously for $b=D$. Consequently, for all $(x,y)\in[0,1]$
\begin{align*}
&\left|\int_0^t\int_0^1\left(P_b(t-s,x,z)-P_b(t-s,y,z)\right)^2d\mu(z)ds\right|^{\frac{1}{2}}\leq 3\cdot 2 L_1^{\frac{1}{2}} t^{\frac{1}{2}}|x-y|^{\frac{1}{2}}.
\end{align*}
We conclude
\begin{align*}
\mathbb{E}&\left(|v_1(t,x)-v_1(t,y)|^q\right)\leq
3^{q}2^{2q-1}T^{\frac{q}{2}}C_qL_1^{\frac{q}{2}}\left(\lVert M\rVert_{q,T}^q+L^q\lVert v_0\rVert_{q,T}^q\right)|x-y|^{\frac{q}{2}}.
\end{align*}
This proves the spacial estimate. \par We now turn to the temporal esimate, where we adapt ideas from \cite[Proposition 4.3]{HYT}. Let $s,t\in[0,T]$ with $s<t$ and $x\in[0,1]$ be fixed. Then, by using the Burkholder-Davis-Gundy inequality, Minkowski's integral inequality and inequality \eqref{festimate}, we get
\begin{align}
&\mathbb{E}\left(|v_1(t,x)-v_1(s,x)|^q\right)\notag\\
&\leq C_q\left|\int_0^t\int_0^1\left|\left(P_b(t-u,x,y)-P_b(s-u,x,y)\mathds{1}_{[0,s]}(u)\right)^2\mathbb{E}\left(f(s,v_0(s,y))^q\right)\right|^{\frac{2}{q}}d\mu(y)du \right|^{\frac{q}{2}}\notag\\
&\leq  2^{q-1}C_q\left(\lVert M\rVert_{q,T}^q+L^q\lVert v_0\rVert_{q,T}^q\right)\left|\int_0^t\int_0^1\left(P_b(t-u,x,y)-P_b(s-u,x,y)\mathds{1}_{[0,s]}(u)\right)^2d\mu(y)du \right|^{\frac{q}{2}}\notag.
\end{align}
We split the above integral in the time intervals $[0,s]$ and $(s,t]$ and consider the first part,
\begin{align}
\int_0^s&\int_0^1\left(P_b(t-u,x,y)-P_b(s-u,x,y)\mathds{1}_{[0,s]}(u)\right)^2d\mu(y)du\notag \\
&= \int_0^s\int_0^1\left(P_b(t-u,x,y)-P_b(s-u,x,y)\right)^2d\mu(y)du\notag.
\end{align}
By Lemma \ref{estimate_resolvent_lemma_2},
\begin{align*}
&\left(\int_0^1 \left(P_b(t-u,x,y) - P_b(s-u,x,y)\right)^2d\mu(y)\right)^{\frac{1}{2}} \\ &- \left( 
  \int_0^1\left(\left\langle P_b(t-u,\cdot,y)-P_b(s-u,\cdot,y),f_n^x\right\rangle\right)^2d\mu(y)\right)^{\frac{1}{2}}\\
&\leq \left( \int_0^1 \left(P_b(t-u,x,y) - P_b(s-u,x,y)- 
 \left(\left\langle P_b(t-u,\cdot,y)-P_b(s-u,\cdot,y),f_n^x\right\rangle\right)\right)^2d\mu(y)\right)^{\frac{1}{2}} \\
&\leq \left(\int_0^1\left(\left\langle P_b(t-u,\cdot,y),f_n^x\right\rangle - P_b(t-u,x,y)\right)^2d\mu(y)\right)^{\frac{1}{2}}  \\ &~~+
\left(\int_0^1\left(\left\langle P_b(s-u,\cdot,y),f_n^x\right\rangle - P_b(s-u,x,y)\right)^2d\mu(y) \right)^{\frac{1}{2}}\\
&\leq C_6^{\frac{1}{2}}(t+s-2u)r_{\max}^{\frac{n}{2}}.
\end{align*}
By resorting and squaring,
\begin{align*}
&\int_0^1 \left(P_b(t-u,x,y) - P_b(s-u,x,y)\right)^2d\mu(y)\\
&\leq 2 
  \int_0^1\left(\left\langle P_b(t-u,\cdot,y)-P_b(s-u,\cdot,y),f_n^x\right\rangle\right)^2d\mu(y)+ 2 C_6(t+s-2u)^2r_{\max}^n
\end{align*}
and by integration,
\begin{align*}
&\int_0^s\int_0^1 \left(P_b(t-u,x,y) - P_b(s-u,x,y)\right)^2d\mu(y)ds\\ 
&\leq 2 
  \int_0^s\int_0^1\left(\left\langle P_b(t-u,\cdot,y)-P_b(s-u,\cdot,y),f_n^x\right\rangle\right)^2d\mu(y)+ C_6(t+s-2u)^2r_{\max}^ndu\\
  &= 2 
 \int_0^s \int_0^1\left(\left\langle P_b(t-u,\cdot,y)-P_b(s-u,\cdot,y),f_n^x\right\rangle\right)^2d\mu(y)du+
\frac{4}{6}C_6\left((t+s)^3-(t-s)^3\right)r_{\max}^n.
\end{align*}
%
Now, let $b=N$. We consider the first term on the right-hand side of the last equality.
Applying the Cauchy-Schwarz inequality,
\begin{align*}
&\left|\int_0^s\int_0^1\left(\left\langle P_N(t-u,z,y)-P_N(s-u,z,y),f_n^x(z)\right\rangle\right)^2d\mu(y)du \right| \\
&\leq \left\lVert f_n^x\right\rVert_{\mu}^2\int_0^s\int_0^1
(t-s)^2+\sum_{k=2}^{\infty}\frac{\left(\sin\left(\sqrt{\lambda_k^N}(t-u)\right)-\sin\left(\sqrt{\lambda_k}(s-u)\right)\right)^2}{\lambda_k^N}\left(\varphi_k^N\right)^2(y)d\mu(y)du.
\end{align*}
Since $\left\lVert \varphi_k^N\right\rVert_{\mu}=1$,
\begin{align*}
&\int_0^s\int_0^1 (t-s)^2+
\sum_{k=2}^{\infty}\frac{\left(\sin\left(\sqrt{\lambda_k^N}(t-u)\right)-\sin\left(\sqrt{\lambda_k^N}(s-u)\right)\right)^2}{\lambda_k^N}\left(\varphi^N_k\right)^2(y)d\mu(y)du\\
&= \int_0^s (t-s)^2+ \sum_{k=2}^{\infty}\frac{\left(\sin\left(\sqrt{\lambda_k^N}(t-u)\right)-\sin\left(\sqrt{\lambda_k^N}(s-u)\right)\right)^2}{\lambda_k^N}du
\\
&= s(t-s)^2+\sum_{k=2}^{\infty}\frac{1}{\lambda_k^N}\int_0^s\left(\sin\left(\sqrt{\lambda_k^N}(t-u)\right)-\sin\left(\sqrt{\lambda_k^N}(s-u)\right)\right)^2du\\
&= s(t-s)^2+\sum_{k=2}^{\infty}\frac{1}{\lambda_k^N}\int_0^s\left(\sin\left(\sqrt{\lambda_k^N}(t-s+u)\right)-\sin\left(\sqrt{\lambda_k^N}(u)\right)\right)^2du\\
&= s(t-s)^2+\sum_{k=2}^{\infty}\frac{1}{\left(\lambda_k^N\right)^{\frac{3}{2}}}\int_0^{\sqrt{\lambda_k^N}s}\left(\sin\left(t-s+u\right)-\sin\left(u\right)\right)^2du\\
&= s(t-s)^2+\sum_{k=2}^{\infty}\frac{1}{\left(\lambda_k^N\right)^{\frac{3}{2}}} \sin^2\left(\frac{t-s}{2}\right)\left(\sin(t-s+2\sqrt{\lambda_k^N}s)-\sin(t-s)+2\sqrt{\lambda_k^N}s\right)
\\&\leq (t-s)^2+
(2+2T)\sum_{k=2}^{\infty}\frac{1}{\lambda_k^N}\sin^2\left(\frac{t-s}{2}\right)\\
&\leq T(t-s)^2+\left(\frac{1}{2}+\frac{1}{2}T\right)\sum_{k=2}^{\infty}\frac{1}{\lambda_k^N}(t-s)^2,
\end{align*}
whereby $\sum_{k=2}^{\infty}\frac{1}{\lambda_k^N}<\infty$ since $\gamma<\frac{1}{2}.$ We turn to the second part and get analogous to the first part
\begin{align*}
&\int_s^t\int_0^1\left(P_N(t-u,x,y)\right)^2d\mu(y)du \\
&\leq 2\int_s^t\int_0^1\left(\left\langle P_N(t-u,\cdot,y),f_n^x\right\rangle\right)^2d\mu(y)  +2 C_6(t-u)^2r_{\max}^ndu\\
&= 2\int_s^t\int_0^1\left(\left\langle P_N(t-u,\cdot,y),f_n^x\right\rangle\right)^2d\mu(y)du  +\frac{2}{3} C_6(t-3)^3r_{\max}^n.
\end{align*}
Again, we give an upper bound for the integral term.
\begin{align*}
&\int_s^t\int_0^1\left(\left\langle P_N(t-u,\cdot,y),f_n^x(\cdot)\right\rangle\right)^2d\mu(y)du  \\
&\leq \left\lVert f_n^x\right\rVert_{\mu}^2\int_s^t\int_0^1
(t-u)^2+\sum_{k=2}^{\infty}\frac{\sin^2\left(\sqrt{\lambda_k^N}(t-u)\right)^2}{\lambda_k^N}\left(\varphi_k^N\right)^2(y)d\mu(y)du.
\end{align*}
With similar methods as above,
\begin{align*}
&\int_s^t\int_0^1 (t-u)^2+
\sum_{k=1}^{\infty}\frac{\sin^2\left(\sqrt{\lambda_k^N}(t-u)\right)}{\lambda_k^N}\left(\varphi_k^N\right)^2(y)d\mu(y)du\\
&=\frac{(t-s)^3}{3}+\sum_{k=2}^{\infty}\int_s^t \frac{\sin^2\left(\sqrt{\lambda_k^N}(t-u)\right)}{\lambda_k^N}du\\
&=\frac{(t-s)^3}{3}+\sum_{k=2}^{\infty}\int_0^{t-s}\frac{\sin^2\left(\sqrt{\lambda_k^N}(u)\right)}{\lambda_k^N}du\\
&\leq\frac{(t-s)^3}{3}+\sum_{k=2}^{\infty}\int_0^{t-s}\frac{\left|\sin\left(\sqrt{\lambda_k^N}u\right|\right)}{\lambda_k^N}du\\
&\leq\frac{(t-s)^3}{3}+\sum_{k=2}^{\infty}\int_0^{t-s}\frac{u}{\sqrt{\lambda_k^N}}du\\
&=\frac{(t-s)^3}{3}+\frac{1}{2}\sum_{k=1}^{\infty} \frac{1}{\sqrt{\lambda_k^N}}(t-s)^2
\end{align*}
with $\sum_{k=2}^{\infty}\frac{1}{\sqrt{\lambda_k^N}}<\infty$.
Further, by \eqref{f_n^x_norm_esimate},
\begin{align*}
\left\lVert f_n^x\right\rVert_{\mu}^2<r_{\min}^{-d_H}r_{\max}^{-nd_H}\nu_{\min}^{-n}.
\end{align*}
Consequently, there exists $C>0$ and $C'>0$ such that for all $t,s\in[0,T],~x\in F$, $n\in\mathbb{N}$
\begin{align*}
\int_0^t\int_0^1\left(P_N(t-u,x,y)-P_N(s-u,x,y)\mathds{1}_{[0,s]}(u)\right)^2d\mu(y)du &\leq C(t-s)^2r_{\max}^{-nd_H}\nu_{\min}^{-n}+C'r_{\max}^n\\
& C(t-s)^2r_{\max}^{-nd_H}\nu_{\min}^{-n}+C''r_{\max}^n,
\end{align*}
where $C''\coloneqq \max\left\{C',C(t-s)^2\left(d_H+\frac{\log(\nu_{\min})}{\log (r_{\max})}\right)\right\}$. In order to find the minimum in $n$, we define
\begin{align*}
f(y)
&\coloneqq C(t-s)^{2}e^{y\log\left(\frac{1}{r_{\max}}\right)\left(d_H+\frac{\log(\nu_{\min})}{\log (r_{\max})}\right)}+C'' e^{-\log\left(\frac{1}{r_{\max}}\right)y}.
\end{align*}
We differentiate:
\begin{align*}
f'(y)&=C(t-s)^2\log\left(\frac{1}{r_{\max}}\right)\left(d_H+\frac{\log(\nu_{\min})}{\log (r_{\max})}\right)e^{y\log\left(\frac{1}{r_{\max}}\right)\left(d_H+\frac{\log(\nu_{\min})}{\log (r_{\max})}\right)}\\
&~~~-C''\log\left(\frac{1}{r_{\max}}\right)e^{-\log\left(\frac{1}{r_{\max}}\right)y}.
\end{align*}
Setting zero we get
\begin{align*}
&e^{y\log\left(\frac{1}{r_{\max}}\right)\left(d_H+\frac{\log(\nu_{\min})}{\log (r_{\max})}+1\right)}\\
&=\frac{C''\log\left(\frac{1}{r_{\max}}\right)}{C(t-s)^2\log\left(\frac{1}{r_{\max}}\right)\left(d_H+\frac{\log(\nu_{\min})}{\log (r_{\max})}\right)}=\frac{C''}{C(t-s)^2\left(d_H+\frac{\log(\nu_{\min})}{\log (r_{\max})}\right)}.
\end{align*}
By logarithmising we obtain
\begin{align*}
y\log\left(\frac{1}{r_{\max}}\right)\left(d_H+\frac{\log(\nu_{\min})}{\log (r_{\max})}+1\right)
=\log\left(\frac{C''}{C(t-s)^2\left(d_H+\frac{\log(\nu_{\min})}{\log (r_{\max})}\right)}\right).
\end{align*}
Solving this equation for $y$ we get
\begin{align*}
y=\frac{1}{\log\left(\frac{1}{r_{\max}}\right)\left(d_H+\frac{\log(\nu_{\min})}{\log (r_{\max})}+1\right)}\log\left(\frac{C''}{C(t-s)^2\left(d_H+\frac{\log(\nu_{\min})}{\log (r_{\max})}\right)}\right),
\end{align*}
which we denote by $y_0$. This value does not need to be an integer, but there exists an integer $n$ with $n\in[y_0,y_0+1)$. 
Since $y_0$ is the unique minimum on $\mathbb{R}$, $f$ is increasing on $[y_o,\infty)$. Hence, there exists $C'''$ such that
\begin{align*}
&\int_0^t\int_0^1\left(P_N(t-u,x,y)-P_N(s-u,x,y)\mathds{1}_{[0,s]}(u)\right)^2d\mu(y)du \\&\leq f\left(y_0+1\right)\\
&= C(t-s)^2\left(\frac{1}{r_{\max}}\right)^{\left(d_H+\frac{\log(\nu_{\min})}{\log(r_{\max})}\right)}
\left(\frac{1}{r_{\max}}\right)^{\frac{\log\left(\frac{C''}{C(t-s)^2\left(d_H+\frac{\log(\nu_{\min})}{\log (r_{\max})}\right)}\right)}{\log\left(\frac{1}{r_{\max}}\right)}\frac{d_H+\frac{\log(\nu_{\min})}{\log(r_{\max})}}{d_H+1+\frac{\log(\nu_{\min})}{\log(r_{\max})}}}\\
&~~~+C''\left(\frac{1}{r_{\max}}\right)^{-\frac{1}{2}}\left(\frac{1}{r_{\max}}\right)^{\frac{\log\left(\frac{C''}{C(t-s)^2\left(d_H+\frac{\log(\nu_{\min})}{\log (r_{\max})}\right)}\right)}{\log\left(\frac{1}{r_{\max}}\right)}\frac{-1}{d_H+\frac{\log(\nu_{\min})}{\log(r_{\max})}+1}}\\
&= C(t-s)^2\left(\frac{1}{r_{\max}}\right)^{\left(d_H+\frac{\log(\nu_{\min})}{\log(r_{\max})}\right)}
\left(\frac{C''}{C(t-s)^2\left(d_H+\frac{\log(\nu_{\min})}{\log (r_{\max})}\right)}\right)^{\frac{d_H+\frac{\log(\nu_{\min})}{\log(r_{\max})}}{d_H+1+\frac{\log(\nu_{\min})}{\log(r_{\max})}}}\\
&~~~+C''\left(\frac{1}{r_{\max}}\right)^{-\frac{1}{2}}\left(\frac{C''}{C(t-s)^2\left(d_H+\frac{\log(\nu_{\min})}{\log (r_{\max})}\right)}\right)^{\frac{-1}{d_H+1+\frac{\log(\nu_{\min})}{\log(r_{\max})}}}\\
&=C'''(t-s)^{\frac{2}{d_H+1+\frac{\log(\nu_{\min})}{\log(r_{\max})}}}.
\end{align*}
The case $b=D$ works similarly.


\end{proof}

\begin{kor}\label{Sqt_korollar}
Let $q\geq 2$ and $v_0\in\mathcal{S}_{q,T}$. Then, $v_i,~ i=1,2,3,$ defined as in \eqref{v1}-\eqref{v3} are elements of $\mathcal{S}_{q,T}$.
\end{kor}

\begin{proof}
By setting $s=0$ in Proposition \ref{stochastic_continuity} we obtain $\left\lVert v_i\right\rVert_{q,T}<\infty$, $i=1,2$. We need to show that $v_1$ is predictable. 
For $n\in\mathbb{N}$ let
\begin{align*}
v_1^n(t,x)=\sum_{i,j=0}^{2^n-1}v_1\left(\frac{i}{2^n}T,\frac{j}{2^n}\right)\mathds{1}_{\left(\frac{i}{2^n}T,\frac{i+1}{2^n}T\right]}(t)\mathds{1}_{\left(\frac{j}{2^n},\frac{j+1}{2^n}\right]}(x), ~~(t,x)\in[0,T]\times[0,1].
\end{align*}
It holds evidently $\left\lVert v_1^n\right\rVert_{q,T}<\infty$.  To prove that $v_1^n$ is predictable, we show that $v_1^n$ is the $\mathcal{S}_{q,T}$-limit of a sequence of simple functions. To this end, let for $N\geq 1$
\begin{align*}
v_1^{n,N}(t,x)=v_1^n(t,x)\wedge N, ~~ t\in[0,T],~ x\in[0,1].
\end{align*}
This defines a simple function since $v_1\left(\frac{i}{2^n}T,\frac{j}{2^n}\right)\wedge N$ is $\mathcal{F}_{\frac{iT}{2^n}}$-measurable and bounded. It converges in $\mathcal{S}_{q,T}$ to $v_1^n$, which can be seen as follows:
\begin{align*}
&\lim_{N\to\infty}\sup_{t\in[0,T]}\sup_{x\in[0,1]}\left\lVert v_1^n(t,x)-v_1^{n,N}(t,x)\right\rVert_{L^q(\Omega)}\\
&\leq  \lim_{N\to\infty}\sup_{t\in[0,T]}\sup_{x\in[0,1]}\sum_{i,j=0}^{2^n-1}\left\lVert v_1\left(\frac{i}{2^n}T,\frac{j}{2^n}\right)- v_1\left(\frac{i}{2^n}T,\frac{j}{2^n}\right)\wedge N\right\rVert_{L^q(\Omega)}\\
&= \lim_{N\to\infty}\sum_{i,j=0}^{2^n-1}\left\lVert v_1\left(\frac{i}{2^n}T,\frac{j}{2^n}\right)- v_1\left(\frac{i}{2^n}T,\frac{j}{2^n}\right)\wedge N\right\rVert_{L^q(\Omega)}
\\ &=0,
\end{align*}
where the last equation follows from the monotone convergence theorem. We conclude that $v_1^n$ is  predictable for $n\in\mathbb{N}$ .
By Proposition \ref{stochastic_continuity}, there exists a constant $C_8'$ such that
\begin{align*}
\left\lVert v_1-v_1^n\right\rVert_{q,T}&\leq \sup_{|s-t|<\frac{T}{n}}\sup_{|x-y|<\frac{1}{n}}\left\lVert v_1(s,x)-v_1(t,y)\right\rVert_{L^q(\Omega)}\\
&\leq  \sup_{|s-t|<\frac{T}{n}}\sup_{|x-y|<\frac{1}{n}}\left\lVert v_1(s,x)-v_1(t,x)\right\rVert_{L^q(\Omega)}\\
&~~~+\sup_{|s-t|<\frac{T}{n}}\sup_{|x-y|<\frac{1}{n}}\left\lVert v_1(t,x)-v_1(t,y)\right\rVert_{L^q(\Omega)}\\
&\leq C_8'\left(\left(\frac{T}{n}\right)^{\frac{1}{2}-\frac{\gamma\delta}{2}}+\left(\frac{1}{n}\right)^{\frac{1}{2}}\right)\to 0, ~n\to\infty.
\end{align*}
Hence, $v_1$ is predictable. The predictability of $v_2$ and $v_3$ follows from the fact that they are measurable and deterministic. 
\end{proof}

\begin{thm}
Assume Condition \ref{hypo} with $q\geq 2$.  Then the SPDE \eqref{spde1} has a unique mild solution in $\mathcal{S}_{q,T}$.
\end{thm}
\begin{proof}
\textit{Uniqueness:} For that, let $u,\widetilde u\in\mathcal{S}_{q,T}$ be mild solution of $\eqref{spde1}$. Then $v\coloneqq u-\widetilde u\in\mathcal{S}_{2,T}$. With $G(t)\coloneqq \sup_{x\in[0,1]}\mathbb{E}\left[v^2(t,x)\right]$ and by using Walsh's isometry,
we calculate for $(t,x)\in[0,T]\times[0,1]$
\begin{align*}
\mathbb{E}\left[v(t,x)^2\right]&=\mathbb{E}\left[\left(\int_0^t\int_0^1P_b(t-s,x,y)\left(f(s,u(s,y))-f\left(s,\widetilde u(s,y)\right)\right)\xi(s,y)d\mu(y)ds\right)^2\right]\\
&=\mathbb{E}\left[\int_0^t\int_0^1P_b(t-s,x,y)^2\left(f(s,u(s,y))-f\left(s,\widetilde u(s,y)\right)\right)^2d\mu(y)ds\right]\\
&\leq L^2\mathbb{E}\left[\int_0^t\int_0^1 v^2(s,y)P_b^2(t-s,x,y)d\mu(y)ds\right]\\
&\leq L^2\left[\int_0^t\sup_{y\in[0,1]}\mathbb{E}\left[v^2(s,y)\right]\int_0^1 P_b^2(t-s,x,y)d\mu(y)ds\right]\\
&\leq L^2\sup_{t\in[0,T]}\left\lVert P_b(t,x,\cdot)\right\rVert_{\mu}^2\int_0^t \sup_{x\in[0,1]}\mathbb{E}\left[v^2(s,x)\right]ds\\
&\leq L^2\sup_{t\in[0,T]}\left\lVert P_b(t,x,\cdot)\right\rVert_{\mu}^2\int_0^t G(s)ds.\\
\end{align*}

	It follows
\begin{align*}
G(t)\leq L^2\sup_{t\in[0,T]}\left\lVert P_b(t,x,\cdot)\right\rVert_{\mu}\int_0^tG(s)ds.
\end{align*}
Since $G$ is continuous on $[0,T]$ (use Proposition \ref{stochastic_continuity} by setting $v_0=v$), we can use Gronwall's lemma to conclude $G(s)=0$ for $s\in[0,T]$ and thus $u(t,x)=\widetilde u(t,x)$ almost surely for every $(t,x)\in[0,T]\times[0,1]$.\\
\textit{Existence:} We follow the methods in \cite[Theorem 7.5]{HYC} and use Picard iteration to find a solution. For that, let $u_2=0\in\mathcal{S}_{q,T}$ and for $n\geq 2$
\begin{align}
u_{n+1}(t,x)=\int_0^1&\frac{\partial}{\partial t}P_b(t,x,y)u_0(y)d\mu(y)+\int_0^1P_b(t,x,y)u_1(y)d\mu(y)\\+&\int_0^t\int_0^1P_b(t-s,x,y)f(s,u_n(s,y))\xi(s,y)d\mu(y)ds. \label{picard_equ}
\end{align}
From Proposition \ref{v_2_continuity} and \ref{stochastic_continuity} it follows that $u_n\in\mathcal{S}_{q,T}$ for every $n\geq 2$. We prove that $(u_n)_{n\in\mathbb{N}}$ is a Cauchy sequence in $\mathcal{S}_{q,T}$. Let $w_n=u_{n+1}-u_n\in\mathcal{S}_{q,T}$. By using the Burkholder-Davis-Gundy inequality, the Lipschitz property of $f$ as well as Minkowski's integral inequality we get 
\begin{align*}
&\mathbb{E}\left[w_{n+1}(t,x)^q\right]\\
&= \mathbb{E}\left[\left|\int_0^t\int_0^1P(t-s,x,y)\left(f(s,u_{n+1}(s,y))-f\left(s,u_n(s,y)\right)\right)\xi(s,y)d\mu(y)ds\right|^{q}\right]\\
&\leq C_q\mathbb{E}\left[\left|\int_0^t\int_0^1P^2(t-s,x,y)\left(f(s,u_{n+1}(s,y))-f\left(s,u_n(s,y)\right)\right)^2d\mu(y)ds\right|^{\frac{q}{2}}\right]\\
&\leq C_q L^q\mathbb{E}\left[\left|\int_0^t\int_0^1 P^2(t-s,x,y)w_n^2(s,y)d\mu(y) ds\right|^{\frac{q}{2}}\right]\\
&\leq  C_q L^q\left(\int_0^t\int_0^1 P^2(t-s,x,y)\left(\mathbb{E}\left[|w_n(s,y)|^q\right]\right)^{\frac{2}{q}}	d\mu(y)ds\right)^{\frac{q}{2}}\\
&\leq  C_q L^q \sup_{t\in[0,T]}\left\lVert P(t,x,\cdot)\right\rVert^{q}_{\mu} \left(\int_0^t\sup_{x\in[0,1]}\left(\mathbb{E}\left[|w_n(s,y)|^q\right]\right)^{\frac{2}{q}}	ds\right)^{\frac{q}{2}}.
\end{align*}
Set  $H_n(t)=\sup_{x\in[0,1]}\left(\mathbb{E}\left[|w_n(t,y)|^q\right]\right)^{\frac{2}{q}}$ for $n\geq 2,~ t\in[0,T]$. Then for every $n\geq 2$ there exists a constant $\kappa_n$ such that $|H_n(t)|\leq \kappa_n$ for every $t\in[0,T]$. With Proposition \ref{propagator_prop} it follows for $(t,x)\in[0,T]\times[0,1]$
\begin{align*}
&\left(\mathbb{E}\left[w_{n+1}(t,x)^q\right]\right)^{\frac{2}{q}}\leq C_q^{\frac{2}{q}}L^2\sup_{t\in[0,T]}\left\lVert P(t,x,\cdot)\right\rVert_{\mu}^2\int_0^t H_n(s)ds
\end{align*}
and thus
\begin{align*}
H_{n+1}(t)\leq C_q^{\frac{2}{q}}L^2 \sup_{t\in[0,T]}\left\lVert P(t,x,\cdot)\right\rVert_{\mu}^2 \int_0^t H_n(s)ds.
\end{align*}
With $\kappa\coloneqq C_q^{\frac{2}{q}}L^2 \sup_{t\in[0,T]}\left\lVert P(t,x,\cdot)\right\rVert_{\mu}^2$ we see that $H_3(t)\leq \kappa\kappa_2t$ and deduce inductively 
\begin{align*}
H_{n+2}(t)\leq \kappa_2\frac{(\kappa t)^{n}}{n!},~~ n\geq 1.
\end{align*}
The series $\sum_{n\geq 3} H_n^{\frac{1}{2}}(t)$ is uniformly convergent on $[0,T]$, which can be verified by the ratio test using that $\sqrt{\frac{H_{n+1}(t)}{H_n(t)}}\leq \sqrt{\frac{\kappa t}{n+1}}$ for $n\geq 2$. We conclude  \[\sup_{t\in[0,T]}\sqrt{H_n(t)}\to 0,~ n\to\infty,\] 
which implies the same for $\left\lVert w_n\right\rVert_{q,T}$. Hence, $(u_n)_{n\geq 2}$ is Cauchy in $\mathcal{S}_{q,T}$ and we denote the limit by $u$. To verify that $u$ satisfies \eqref{mild_solution_wave} we take the limit in $L^q(\Omega)$ for $n\to\infty$ on both sides of \eqref{picard_equ} for every $(t,x)\in[0,T]\times[0,1]$. We get $u(t,x)$ on the left-hand side for any $(t,x)\in[0,T]\times[0,1]$. For the right-hand side we note that for $(t,x)\in[0,T]\times[0,1]$
\begin{align*}
&\mathbb{E}\left[\left|\int_0^t\int_0^1P_b(t-s,x,y)\left(f(s,u(s,y))-f\left(s,u_n(s,y)\right)\right)\xi(s,y)d\mu(y)ds\right|^{q}\right]\\
&\leq  C_q L^q\left(\int_0^t\int_0^1 P_b^2(t-s,x,y)\left(\mathbb{E}\left[|u(s,y)-u_n(s,y)|^q\right]\right)^{\frac{2}{q}}	d\mu(y)ds\right)^{\frac{q}{2}},
\end{align*}
which goes to zero as $n$ tends to infinity with the same argumentation as before.
\end{proof}

We have computed different temporal Hölder exponents. The following lemma shows which one is greater.

\begin{lem} \label{hoelder_exponent_lemma}
Let $r_1,...,r_N$ and $\mu_1,...,\mu_N$ be arbitrary, but chosen according to the conditions given in section \ref{section_1}. Then,
\begin{align}
\left(d_H+1+\frac{\log \nu_{\min}}{\log r_{\max}}\right)^{-1}\leq 2-(2+\delta)\gamma.\notag
\end{align}
\end{lem}
\begin{proof}
We have
\begin{align}
&\frac{\min_{i=1,...,N}\log\mu_i-\log r_i^{d_H}}{\max_{i=1,...,N}\log r_i}+d_H+1\notag\\
&~~~~=\frac{\min_{i=1,...,N}\log\mu_i-\log r_i^{d_H}}{\max_{i=1,...,N}\log r_i}-(1-d_H)+2\notag\\
&~~~~\geq \max_{i=1,...,N}\frac{\log\mu_i-\log r_i^{d_H}}{\log r_i}-(1-d_H)+2\notag\\
&~~~~= \max_{i=1,...,N}\frac{\log\mu_i-\log r_i+(1-d_H)\log r_i}{\log r_i}-(1-d_H)+2\notag\\
&~~~~= \max_{i=1,...,N}\frac{\log\mu_i-\log r_i}{\log r_i}+2.\notag
\end{align}
Using that as well as the fact that $\gamma<\frac{1}{2}$,
\begin{align*}
\left(d_H+1+\frac{\log \nu_{\min}}{\log r_{\max}}\right)^{-1}
&\leq \left(\max_{i=1,...,N}\frac{\log\mu_i-\log r_i}{\log r_i}+2\right)^{-1}\\
&=\min_{i=1,...,N}\left(\frac{\log\mu_i-\log r_i}{\log r_i}+2\right)^{-1}\\
&=\min_{i=1,...,N} \frac{\log r_i}{\log\mu_i+\log r_i}\\
&=\min_{i=1,...,N}\left(1-\frac{\log \mu_i}{\log\mu_i+\log r_i}\right)\\
&= \left(1-\max_{i=1,...,N}\frac{\log \mu_i}{\log\mu_i+\log r_i}\right)\\
&= \left(1-\gamma\delta\right)\\
&< 2-2\gamma-\gamma\delta.
\end{align*}
\end{proof}

Using this lemma and the established continuity properties (compare Proposition \ref{v_2_continuity} and Proposition \ref{stochastic_continuity}), the main result of this paper, Theorem \ref{main_theorem}, is a direct consequence of Kolmogorov's continuity theorem.

\begin{figure}[t]
\centering
\includegraphics[scale=0.4]{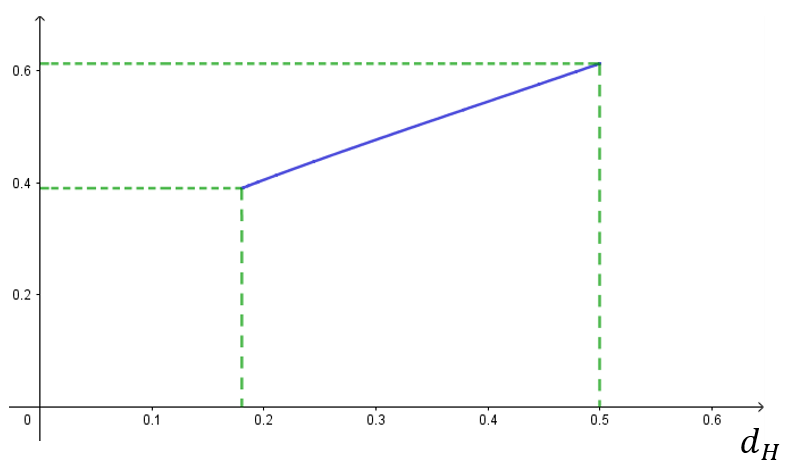}
\caption{Hölder exponent graphs}
\label{hoelder_plot_2}
\end{figure}

\begin{exa}
We discussed the case of $\mu$ being the natural on a given Cantor-like set in Section \ref{section_1}.
If $\mu$ is not the natural measure on a given Cantor-like set, then
$\nu_{\min}$ does not vanish. 
As an example, consider the Classical Cantor set with weights $\mu_1,\mu_2\in(0,1)$.  If $u_0$, $u_1$ and $f$ satisfy Assumption \ref{hypo} and $f$ is uniformly bounded, $q$ can be chosen arbitrarily large. We then have
as ess. temporal Hölder exponent
\begin{align*}
\frac{1}{d_H+1+\frac{\log(\nu_{\min})}{\log(r_{\max)}}} = \frac{1}{\frac{\log 2}{\log 3}+1-\frac{\log\mu_{\min}+\log 2}{\log 3} } = \frac{1}{1-\frac{\log(\mu_{\min})}{\log 3}},
\end{align*}
 provided that $\delta+1<\frac{1}{\gamma}$. This is satisfied if \[\max_{i=1,2} \frac{\log \mu_i}{\log \frac{\mu_i}{3}}+\frac{\log 2}{\log 6}<1,\] which holds for $\mu_1,\mu_2$ such that $\min_{i=1,2}\mu_i>0.18$.
The behaviour of the temporal Hölder exponent is visualized on the right-hand side of Figure \ref{hoelder_plot_2}.
\end{exa}

\subsection{Intermittency}
Let $\varepsilon\geq 0$. According to \cite{CJI} we call the mild solution of a stochastic wave equation $u$ weakly intermittent on $[\varepsilon,1-\varepsilon]$ if for the upper moment Lyapunov exponents
\begin{align*}
\bar\gamma(p,x)\coloneqq \limsup_{t\to\infty}\frac{1}{t}\log\mathbb{E}\left[u(t,x)^2\right]
\end{align*}
it holds 
\begin{align*}
\bar\gamma(2,x)>0,~~~ \bar\gamma(p,x)<\infty, ~~ x\in[\varepsilon,1-\varepsilon], ~p\in [2,\infty).
\end{align*}

In this section we make the following additional assumption:
\begin{hyp}
We assume Condition \ref{hypo} with $q\geq 2$ and that $f$ fulfills the following Lipschitz and linear growth condition: For all $(w,t,x)\in\Omega\times[0,T]\times\mathbb{R}$ there exists a constant $L>0$ such that
\begin{align*}
|f(\omega,t,x)-f(\omega,t,y)|&\leq L|x-y|,\\
|f(\omega,t,y)|&\leq L(1+|x|).
\end{align*}
\end{hyp}
\begin{satz}
Let $p\geq 1$. Then there exists constants $C_9,C_{10}>0$ such that for $(t,x)\in[0,\infty)\times[0,1]$
\begin{align*}
\mathbb{E}\left[\left|u(t,x)\right|^p\right]\leq C_9e^{C_{10}p^2t}.
\end{align*}
\end{satz}
\begin{proof}
$v_2$ and $v_3$ are uniformly bounded on $[0,\infty)\times[0,1]$. This can be verified with the same methods as in the proof of Proposition \ref{stochastic_continuity}. For example, for $b=D$ and $(t,x)\in[0,\infty)\times[0,1]$
\begin{align*}
\left|\sum_{k=1}^{\infty}\frac{\sin\left(\sqrt{\lambda_k^D}t\right)}{\sqrt{\lambda_k^D}}\varphi_k^D(x)u_{1,k}^D\right|
&\leq \sum_{k=1}^{\infty} C_0^{-\frac{1}{2}}C_2C_5 k^{-\frac{1}{2\gamma}}k^{\frac{1}{2}}k^{-\frac{1}{2\gamma}}
\leq C_0^{-\frac{1}{2}}C_2C_5\sum_{k=1}^{\infty}  k^{\frac{1}{2}-\frac{1}{\gamma}}<\infty.
\end{align*}

 Hence, there exists a constant $K>0$ such that for $i\in\{2,3\}$, $(t,x)\in[0,\infty)\times[0,1]$ $v_i(t,x)\leq K, i=2,3$.
It follows by using the Burkholder-Davis-Gundy inequality as well as Minkowski's integral inequality,
\begin{align*}
&e^{-\alpha t}\left(\mathbb{E}\left[\left|u(t,x)\right|^p\right]\right)^{\frac{1}{p}}\\
&\leq e^{-\alpha t}2K+2\sqrt{p}\left(\mathbb{E}\left[\left|\int_0^t\int_0^1 e^{-\alpha t}P_b(t-s,x,y)f(s,u(s,y))\xi(s,y)d\mu(y) ds\right|^p\right]\right)^{\frac{1}{p}}\\
&\leq e^{-\alpha t}2K+2\sqrt{p}\left(\int_0^t\int_0^1 e^{-2\alpha t}P_b^2(t-s,x,y)\left(\mathbb{E}\left[|f(s,u(s,y))|^p\right]\right)^{\frac{2}{p}}d\mu(y) ds\right)^{\frac{1}{2}}\\
&\leq e^{-\alpha t}2K+L2\sqrt{p}\left(\int_0^t\int_0^1 e^{-2\alpha t}P_b^2(t-s,x,y)\sup_{z\in[0,1]}\left(1+\left(\mathbb{E}\left[|u(s,z))|^p\right]\right)^{\frac{1}{p}}\right)^2d\mu(y) ds\right)^{\frac{1}{2}}\\
&\leq e^{-\alpha t}2K+L2\sqrt{p}\left(\int_0^t\int_0^1 e^{-2\alpha (t-s)}P_b^2(t-s,x,y)\sup_{z\in[0,1]}\left(e^{-\alpha s}+e^{-\alpha s}\left(\mathbb{E}\left[|u(s,z))|^p\right]\right)^{\frac{1}{p}}\right)^2d\mu(y) ds\right)^{\frac{1}{2}}\\
&\leq e^{-\alpha t}2K\\
&~~~+L2\sqrt{p}\sup_{s,z\in[0,T]\times[0,1]}\left(e^{-\alpha s}+e^{-\alpha s}\left(\mathbb{E}\left[|u(s,z))|^p\right]\right)^{\frac{1}{p}}\right)\left(\int_0^t\int_0^1 e^{-2\alpha (t-s)}P_b^2(t-s,x,y)d\mu(y) ds\right)^{\frac{1}{2}}\\
&\leq e^{-\alpha t}2K+L2\sqrt{p}\left(1+\sup_{s,z\in[0,T]\times[0,1]}e^{-\alpha s}\left(\mathbb{E}\left[|u(s,z))|^p\right]\right)^{\frac{1}{p}}\right)\left(C_3^2\int_0^te^{-2\alpha (t-s)} ds\right)^{\frac{1}{2}}\\
&\leq e^{-\alpha t}2K+\frac{C_3^2L2\sqrt{p}}{\sqrt{2\alpha}}\left(1+\sup_{s,z\in[0,T]\times[0,1]}e^{-\alpha s}\left(\mathbb{E}\left[|u(s,z))|^p\right]\right)^{\frac{1}{p}}\right).\\
\end{align*}
Choose $\alpha= 8C_3^4L^2p$. Then it follows
\begin{align*}
\left(\mathbb{E}\left[\left|u(t,x)\right|^p\right]\right)^{\frac{1}{p}}\leq 4K+e^{\alpha t}=4K+e^{8C_4^4L^2pt}.
\end{align*}
For $p\in[1,2)$ we have for $(t,x)\in[0,\infty)\times[0,1]$
\begin{align*}
\left(\mathbb{E}\left[\left|u(t,x)\right|^p\right]\right)^{\frac{1}{p}}\leq
\left(\mathbb{E}\left[\left|u(t,x)\right|^2\right]\right)^{\frac{1}{2}}\leq
4K+e^{16C_3^4L^2t}\leq
4K+e^{16C_3^4L^2pt}.
\end{align*}
\end{proof}
From the above proposition, it follows immediately for $p\geq 1$
\begin{align*}
\bar\gamma(p)=\limsup_{t\to\infty}\frac{1}{t}\sup_{x\in[0,1]}\log\mathbb{E}\left[|u(t,x)|^p\right]
&\leq \limsup_{t\to\infty} \frac{\log C_9}{t}+C_{10}p^2.
\\& =C_{10}p^2.
\end{align*}
\pagebreak
\begin{satz}
Assume $\displaystyle\inf_{x\in[0,1]}\left|f(x)/x\right|>0.$
\begin{enumerate}
\item Let $b=N$, $\inf_{x\in[0,1]}u_0(x)>0$ and $\inf_{x\in[0,1]}u_1(x)>0$. Then, there exists a constant $\kappa$ such that $\bar\gamma(2,x)\geq\kappa$ for all $x\in[0,1]$. 
\item Let $b=D$, $\varepsilon>0$, $\inf_{x\in[\varepsilon,1-\varepsilon]}u_0(x)>0$ and $\inf_{x\in[\varepsilon,1-\varepsilon]}u_1(x)>0$. Then, there exists a constant $\kappa_{\varepsilon}$ such that $\bar\gamma(2,x)\geq\kappa_{\varepsilon}$ for all $x\in[\varepsilon,1-\varepsilon]$. 
\end{enumerate}
\end{satz}

\begin{proof}
Let $\varepsilon\geq 0$, $\inf_{x\in[\varepsilon,1-\varepsilon]}u_0(x)>0$, $\inf_{x\in[\varepsilon,1-\varepsilon]}u_1(x)>0$ and $x\in[\varepsilon,1-\varepsilon]$. It suffices to find a constant $\beta_{\epsilon}>0$ 
\begin{align}
\int_0^{\infty}e^{-\beta t}\mathbb{E}\left[u(t,x)^2\right]dt=\infty~ \text{ for all } ~\beta\leq\beta_{\epsilon} \label{intermittency_proof}
\end{align}
(see the proof of \cite[Theorem 3.3]{CJI}). 
By using Walsh's isometry and the zero-mean property of the stochastic integral we get
\begin{align*}
\mathbb{E}\left[u(t,x)^2\right]dt
&=\left(v_2(t,x)+v_3(t,x)\right)^2+\int_0^t\int_0^1 P_b^2(t-s,x,y)\mathbb{E}\left[f(u(s,y))^2\right]d\mu(y)ds\\
&~~~+\left(v_2(t,x)+v_3(t,x)\right)^2\mathbb{E}\left[\int_0^t\int_0^1P_b(t-s,x,y)f(u(s,y))\xi(s,y)d\mu(y)ds\right]\\
&=\left(v_2(t,x)+v_3(t,x)\right)^2+\int_0^t\int_0^1 P_b^2(t-s,x,y)\mathbb{E}\left[f(u(s,y))^2\right]d\mu(y)ds
\end{align*}
and thus, by Laplace transformation,
\begin{align*}
&\int_0^{\infty}e^{-\beta t}\mathbb{E}\left[u(t,x)^2\right]dt \\
&= \int_0^{\infty}e^{-\beta t}\left(v_2(t,x)+v_3(t,x)\right)^2dt+ \int_0^{\infty}e^{-\beta t}\int_0^t\int_0^1 P_b^2(t-s,x,y)\mathbb{E}\left[f(u(s,y))^2\right]d\mu(y)dsdt.
\end{align*}
In order to bound the first term on the right-hand side from below, we note that $v_2(0,x)=u_1(x)\geq\inf_{x\in[\varepsilon,1-\varepsilon]}u_1(x)> 0$ and $v_3(0,x)=u_0(x)\geq\inf_{x\in[\varepsilon,1-\varepsilon]}u_0(x)> 0$. Using that both functions are Hölder-continuous in $t$ uniformly for all $x\in[0,1]$ (see Proposition \ref{v_2_continuity}), we obtain the existence of a constant $t_{\varepsilon}>0$ such that 
\begin{align*}
v_3(t,x)>\frac{u_0}{2},  ~~ v_2(t,x)>-\frac{u_0}{4}, ~~ t\in[0,t_{\varepsilon.}] 
\end{align*}
We see that $v_2(t,x)+v_3(t,x)>\frac{u_0}{4}$ for all $(t,x)\in[0,t_{\varepsilon}]\times[\varepsilon,1-\varepsilon]$. It follows that  for all $\beta>0$ there exists a constant $K_{\beta,\varepsilon}$ such that
\begin{align*}
&\int_0^{\infty}e^{-\beta t}\mathbb{E}\left[u(t,x)^2\right]dt \\
&\geq  K_{\beta,\varepsilon} + L_{\varepsilon}^2\int_0^{\infty}e^{-\beta t}\int_0^t\int_0^1 P_b^2(t-s,x,y)\mathbb{E}\left[(u(s,y)^2\right]d\mu(y)dsdt,
\end{align*}
where $K_{\beta,\varepsilon} = \frac{u_0^2}{16\beta}$.
Further, for $(x,y,t)\in[0,1]^2\times[0,\infty)$
\begin{align*}
\int_0^t P_b^2(t-s,x,y)\mathbb{E}\left[(u(s,y)^2\right]d\mu(y)ds=\left(P_b(\cdot,x,y)\ast \mathbb{E}\left[u(\cdot,y)^2\right]\right)(t),
\end{align*}
where $\ast$ denotes the time convolution. It holds $\mathcal{L}_{\beta}(f\ast g)=\mathcal{L}_{\beta}f\cdot\mathcal{L}_{\beta}g$, where $\mathcal{L}$ denotes the Laplace transformation. We thus see 
\begin{align*}
&\int_0^{\infty}e^{-\beta t}\int_0^t\int_0^1 P_b^2(t-s,x,y)\mathbb{E}\left[(u(s,y)^2\right]d\mu(y)dsdt\\
&=\int_0^1\int_0^{\infty}e^{-\beta t}\int_0^t P_b^2(t-s,x,y)\mathbb{E}\left[(u(s,y)^2\right]dsdtd\mu(y)\\
&= \int_0^1 \int_0^{\infty}e^{-\beta t}P_b^2(t,x,y)dt\int_0^{\infty}e^{-\beta s}\mathbb{E}\left[u(s,y)^2\right]dsd\mu(y).
\end{align*}
With $M_{\beta}(x)\coloneqq\int_0^{\infty}e^{-\beta s}\mathbb{E}\left[u(s,x)^2\right]ds $ it follows
\begin{align}
M_{\beta}(x)\geq K_{\beta,\varepsilon} + L_{\varepsilon}^2\int_0^1 \int_0^{\infty}e^{-\beta t}P^2_b(t,x,y)M_{\beta}(y)dtd\mu(y).\label{m_beta_1}
\end{align}
If $b=N$, it holds for all for all $t\geq 0$
\begin{align*}
\left\lVert P_N(t,x,\cdot)\right\rVert_{\mu}^2&= t^2+\sum_{k\geq 2}\frac{\sin^2\left(\sqrt{\lambda_k^N t}\right)}{\lambda_k^N}\left(\varphi_k^N\right)^2(x)\\&\geq t^2
\end{align*}
and thus
\begin{align*}
\int_0^{\infty}\int_0^1e^{-\beta t} P_N^2(t,x,y) K_{\beta,\varepsilon}d\mu(y)dt& =K_{\beta,\varepsilon}\int_0^{\infty}e^{-\beta t}\left\lVert P_N(t,x,\cdot)\right\rVert_{\mu}^2dt\\
&\geq K_{\beta,\varepsilon}\int_0^{\infty}e^{-\beta t}t^2 dt\\
&=  2K_{\beta,\varepsilon}\beta^{-3}.
\end{align*}
By iterating this in \eqref{m_beta_1} we obtain
\begin{align*}
M_{\beta}(x)\geq K_{\beta,\varepsilon}\sum_{n=0}^{\infty}\left(2\beta^{-3}\right)^n.
\end{align*}
This sum diverges if and only if $\beta\leq \sqrt[3]{2}$. Hence, we have shown \eqref{intermittency_proof}. \par 

If $b=D$, we define $c_{\varepsilon}\coloneqq  \inf_{x\in[\varepsilon,1-\varepsilon]}\varphi_1^D(x)$ and calculate
\begin{align*}
\int_0^{\infty}\int_0^1e^{-\beta t} P_D^2(t,x,y) K_{\beta,\varepsilon}d\mu(y)dt &\geq K_{\beta,\varepsilon}\int_0^{\infty}e^{-\beta t}\sum_{k=1}^{\infty}\frac{\sin^2\left(\sqrt{\lambda_k^D t}\right)}{\lambda_k^D}\left(\varphi_k^D\right)^2(x) dt\\
&\geq K_{\beta,\varepsilon}\int_0^{\infty}e^{-\beta t}\frac{\sin^2\left(\sqrt{\lambda_1^D }t\right)}{\lambda_1^D}\left(\varphi_1^D\right)^2(x) dt\\
&\geq K_{\beta,\varepsilon}\int_0^{\infty}e^{-\beta t}\frac{\sin^2\left(\sqrt{\lambda_1^D }t\right)}{\lambda_1^D}c_{\varepsilon}dt\\
&= \frac{K_{\beta,\varepsilon}c_{\varepsilon}}{(\lambda_1^D)^{\frac{3}{2}}}\int_0^{\infty}e^{-\frac{\beta}{\sqrt{\lambda_1^D}} t}\sin^2(t)dt >0\\
&= \frac{K_{\beta,\varepsilon}c_{\varepsilon}}{(\lambda_1^D)^{\frac{3}{2}}}\frac{2}{\left(\frac{\beta}{\sqrt{\lambda_1^D}}\right)^3+4\left(\frac{\beta}{\sqrt{\lambda_1^D}}\right)   }>0.
\end{align*}
The last term is strictly positive, since $\varphi_1^D(x)>0$ for all $x\in(0,1)$ (see \cite[Proposition 2.5]{FLZ}) and therefore bounded from below for all $x\in[\varepsilon,1-\varepsilon]$.
By iterating this in \eqref{m_beta_1} we obtain
\begin{align*}
M_{\beta}(x)\geq K^{'}_{\beta,\varepsilon}\sum_{n=0}^{\infty}\left(\frac{2c_{\varepsilon}\left(\lambda_1^D\right)^{-\frac{3}{2}}}{\left(\frac{\beta}{\sqrt{\lambda_1^D}}\right)^3+4\left(\frac{\beta}{\sqrt{\lambda_1^D}}\right)  }\right)^n.
\end{align*}
Let $\bar\beta\coloneqq \frac{\beta}{\sqrt{\lambda_1^D}}$. 
 The above sum is equal to $\infty$ for all $\beta$ such that $\bar\beta^3+4\bar\beta\leq 2c_{\varepsilon}\left(\lambda_1^D\right)^{-\frac{3}{2}}$. This verifies \eqref{intermittency_proof}.
\end{proof}

\appendix
\section{Some Technical Details}

\begin{lem} \label{appendix_restriction_supp}
Let $b\in\{N,D\}$, $\psi: L^2([0,1],\mu)\to L^2(\supp(\mu),\mu)$, $u\to \left.u\right|_{\supp(\mu)}$ and
\[ \widetilde\Delta_{\mu}^b:~ \psi\left(\mathcal{D}\left(\Delta_{\mu}^b\right)\right)\to L^2(\supp(\mu),\mu),~~
  u\to\psi\circ\Delta_{\mu}^b\circ \psi^{-1}u.\] Then,
\begin{enumerate}[label=(\roman*)]
\item 
   $\widetilde\Delta_{\mu}^b$ is self-adjoint, dissipative and has eigenvalues $\lambda_k^b$ with eigenfunctions $\psi\varphi_k^b$, $k\in\mathbb{N}$. In particular, $\widetilde\Delta_{\mu}^b$ is the generator of a unique strongly continuous semigroup $\left(\widetilde T_t^b\right)_{t\geq 0}$. 
\item   
 $\widetilde\mE\left(\widetilde u,\widetilde v\right)\coloneqq\mE(\psi^{-1}\widetilde u,\psi^{-1}\widetilde v)$, $\widetilde u,\widetilde v\in\widetilde{\mathcal{F}}\coloneqq\psi(\mathcal{F})$ defines a Dirichlet form which is associated to $\widetilde\Delta_{\mu}^N$ and $\widetilde\mE\left(\widetilde u,\widetilde v\right), \widetilde u,\widetilde v\in\widetilde{\mathcal{F}_0}\coloneqq\psi(\mathcal{F}_0)$ defines a Dirichlet form associated to $\widetilde\Delta_{\mu}^D$.
\end{enumerate}    
\end{lem}
\begin{proof}
\begin{enumerate}[label=(\roman*)]
\item
First, we show that $\widetilde\Delta_{\mu}^b$ is self-adjoint.  We denote the inner product on $L^2(\supp(\mu),\mu)$ also by $\langle \cdot,\cdot\rangle_{\mu}$.
Since $\mathcal{D}\left(\Delta_{\mu}^b\right)$ is dense in $L^2([0,1],\mu)$, for any $u\in L^2([0,1],\mu)$ there exists a sequence $(u_n)_{n\in\mathbb{N}}$ with $u_n\in \mathcal{D}\left(\Delta_{\mu}^b\right),$  $n\in\mathbb{N}$ such that $\left\lVert u_n-u\right\rVert_{\mu}\to 0$ for  $n\to\infty$. From $\left\lVert u_n-u\right\rVert_{\mu}=\left\lVert \psi u_n-\widetilde u\right\rVert_{\mu}$ for all $n\in\mathbb{N}$ and $\psi u_n\in \mathcal{D}\left(\widetilde\Delta_{\mu}^b\right)=\psi\left(\mathcal{D}\left(\Delta_{\mu}^b\right)\right)$  the density of $\mathcal{D}\left(\widetilde \Delta_{\mu}^b\right)$ in $L^2(\supp(\mu),\mu)$ follows. Now, let $\widetilde u, \widetilde v\in\mathcal{D}\left(\widetilde\Delta_{\mu}^b\right)=\psi\left(\mathcal{D}\left(\Delta_{\mu}^b\right)\right)$, i.e. there exist unique $u,v\in \mathcal{D}\left(\Delta_{\mu}^b\right)$ such that $\widetilde u=\psi u,~\widetilde v=\psi v$. It is straight forward to check that $v\to\left\langle u,\Delta_{\mu}^bv\right\rangle_{\mu}$ is a linear continuous mapping on $\mathcal{D}\left(\Delta_{\mu}^b\right)$ if and only if $\widetilde v\to\left\langle \widetilde u,\widetilde \Delta_{\mu}^b\widetilde v\right\rangle_{\mu}$ is linear and continuous on $\mathcal{D}\left(\widetilde \Delta_{\mu}^b\right)$, which yields $\mathcal{D}\left(\widetilde \Delta_{\mu}^b\right)$=$\mathcal{D}\left(\left(\widetilde \Delta_{\mu}^b\right)^*\right)$.
Further, for all $\widetilde u, \widetilde v\in \mathcal{D}\left(\widetilde \Delta_{\mu}^b\right)$
\begin{align*}
\left\langle \widetilde\Delta_{\mu}^b\widetilde u,\widetilde v\right\rangle_{\mu}&=
\left\langle \psi\circ \Delta_{\mu}^b\circ\psi^{-1}\circ\psi u,\psi v \right\rangle_{\mu}\\
&= \left\langle \psi\circ \Delta_{\mu}^b u,\psi v \right\rangle_{\mu}\\
&= \left\langle \Delta_{\mu}^b  u, v \right\rangle_{\mu}\\
&= \left\langle  u, \Delta_{\mu}^b v \right\rangle_{\mu}\\
&= \left\langle \psi u, \psi\circ \Delta_{\mu}^b\circ\psi^{-1}\circ\psi v \right\rangle_{\mu}\\
&= \left\langle \widetilde u, \widetilde\Delta_{\mu}^b\widetilde v\right\rangle_{\mu}.
\end{align*}
The self-adjointness of $\Delta_{\mu}^b$ follows. For the dissipativity of $\widetilde \Delta_{\mu}^b$ we obtain from the dissipativity of $\Delta_{\mu}^b$
\begin{align*}
\left\langle \widetilde \Delta_{\mu}^b\widetilde u,\widetilde u\right\rangle_{\mu}=
\left\langle \Delta_{\mu}^b u, u\right\rangle_{\mu}\leq 0.
\end{align*}
The self-adjointness along with the dissipativity implies that $\widetilde \Delta_{\mu}^b$ generates a strongly continuous semigroup $\left(\widetilde T_t^b\right)_{t\geq 0}$ (see \cite[Theorem B.2.2]{KA}). It remains to show that eigenvalues and eigenfunctions of $\widetilde \Delta_{\mu}^b$ and $\Delta_{\mu}^b$ coincide. For that, let $\lambda<0$, $u\in\mathcal{D}\left(\Delta_{\mu}^b\right)$. The bijectivity of $\psi$ implies that $\left(\Delta_{\mu}^b-\lambda\right)u=0$ if and only if $\psi\left(\Delta_{\mu}^b-\lambda\right)u=0$. The results about eigenvalues and eigenfunctions follow.
\item Let $b=N$. Again, let let $\widetilde u, \widetilde v\in\mathcal{D}\left(\widetilde\Delta_{\mu}^N\right)= \psi\left(\mathcal{D}\left(\Delta_{\mu}^N\right)\right)$, i.e. there exist $u,v\in \mathcal{D}\left(\widetilde\Delta_{\mu}^b\right)$ with $\widetilde u=\psi u,~\widetilde v=\psi v$. The density of $\widetilde{\mathcal{F}}$ in $L^2(\supp(\mu),\mu)$ can be checked exactly like the density of $\mathcal{D}\left(\widetilde \Delta_{\mu}^N\right)$ in $\mathcal{H}$. Further, it is obvious that $\widetilde\mE$ defines a positive definite, symmetric bilinear form. Moreover, with $\alpha>0$ and $\widetilde\mE_{\alpha}\left(\widetilde u,\widetilde v\right)\coloneqq \widetilde\mE\left(\widetilde u,\widetilde v\right)+\alpha\left\langle\widetilde u, \widetilde v\right\rangle_{\mu}$, $\left(\widetilde F, \widetilde\mE_{\alpha}\right)$ is a Hilbert space. To verify this, note that $\widetilde\mE_{\alpha}\left(\widetilde u,\widetilde v\right)=\mE_{\alpha}(u,v)$, which implies that $\widetilde\mE_{\alpha}$ defines an inner product. Now, let $\widetilde u_n,~ n\in\mathbb{N}$ be a Cauchy sequence in $\widetilde{\mathcal{F}}$. Then, $u_n=\psi^{-1}\widetilde u_n,~ n\in\mathbb{N}$ is a Cauchy sequence in $\mathcal{F}$ with limit, say $u$. Since 
$\left\lVert \widetilde u_n-\psi u\right\rVert=\left\lVert u_n-u\right\rVert $ for all $n$, $\psi u$ is the limit of $\left(\widetilde u_n\right)_{n\in\mathbb{N}}$ in $\widetilde{\mathcal{F}}$. For the Markov property, we calculate \[\widetilde{\mE}\left(0\vee \widetilde u\wedge 1\right)=\mE\left(0\vee u\wedge 1\right)\leq \mE(u)=\widetilde{\mE}\left(\widetilde u\right).\]
To verify that $\widetilde{\Delta}_{\mu}^N$ is associated to $\widetilde\mE$, we apply the correspondence between $\Delta_{\mu}^N$ and $\mE$ to get for 
\begin{align*}
-\left\langle \widetilde{\Delta}_{\mu}^N\widetilde u,\widetilde v\right\rangle_{\mu}=
-\left\langle \Delta_{\mu}^N u, v\right\rangle_{\mu}
=\mE(u,v)=\widetilde \mE\left(\widetilde u,\widetilde v\right).
\end{align*}
The case $b=D$ works similarly.
\end{enumerate}
\end{proof}

\end{document}